\documentclass{amsart}
\usepackage[normalem]{ulem}
\usepackage{amssymb,epsfig,mathrsfs,mathpazo,scalerel,url,amsthm,thmtools,bbm}
\usepackage{xcolor}
\usepackage{stmaryrd}
\usepackage{epsfig,bm}
\usepackage[multiple]{footmisc}
\usepackage{accents} 
\usepackage{mathtools} 
\newcommand{\osc}{{\rm osc}}

\usepackage{arydshln, booktabs, makecell, pbox, tabularx}

\usepackage{graphicx, caption, subcaption}

\newtheorem{theorem}{Theorem}[section]
\newtheorem{lemma}[theorem]{Lemma}
\newtheorem{proposition}[theorem]{Proposition}

\declaretheorem[style=remark,qed=$\vartriangle$,sibling=theorem]{remark}
\numberwithin{equation}{section}

\newcommand{\eps}{\varepsilon}

\newcommand{\R}{\mathbb R}
\newcommand{\C}{\mathbb C}

\newcommand{\N}{\mathbb N}

\newcommand{\cF}{\mathcal F}

\newcommand{\cL}{\mathcal L}
\newcommand{\cS}{\mathcal S}
\newcommand{\Lis}{\cL\mathrm{is}}

\newcommand{\identity}{\mathrm{Id}}

\DeclareMathOperator{\ran}{ran}
\DeclareMathOperator{\supp}{supp}

\DeclareMathOperator{\diam}{diam}
\DeclareMathOperator*{\argmin}{argmin}

\DeclareMathOperator{\divv}{div}
\DeclareMathOperator{\divvv}{\new{div}}

\newcommand{\new}[1]{{\color{black}{#1}}}

\newcommand{\be}{\begin{equation}}
\newcommand{\ee}{\end{equation}}

\newcommand{\tria}{{\mathcal T}}
\newcommand{\RT}{\mathit{RT}}

{\par\begin{samepage}%
\begin{tabbing}\ttfamily%
 \hspace*{5mm}\=\hspace{3ex}\=\hspace{3ex}\=\hspace{3ex}\=\hspace{3ex}%
\=\hspace{3ex}\=\hspace{3ex}\=\hspace{3ex}\=\hspace{3ex}\kill}%
{\end{tabbing}\end{samepage}}

\makeatletter
\@namedef{subjclassname@2020}{%
  \textup{2020} Mathematics Subject Classification}
\makeatother

\title[A pollution-free ultra-weak FOSLS discretization of the Helmholtz equation]{A pollution-free ultra-weak FOSLS discretization of the Helmholtz equation\\
{\textnormal{\lowercase{\small \uppercase{D}edicated to \uppercase{P}rofessor \uppercase{L}eszek \uppercase{F}.~\uppercase{D}emkowicz on the occasion of his 70$^\mathrm{th}$ birthday}}}}

\date{\today}

\author{Harald Monsuur, Rob Stevenson}

\address{
Korteweg--de Vries (KdV) Institute for Mathematics, University of Amsterdam, P.O. Box 94248, 1090 GE Amsterdam, The Netherlands.
}
\email{h.monsuur@uva.nl, rob.p.stevenson@gmail.com}

\thanks{This research has been supported by the Netherlands Organization for Scientific Research (NWO) under contract.~no.~SH-208-11, and by the
NSF Grant DMS ID 1720297}


\begin{document}

\begin{abstract} 
We consider an ultra-weak first order system discretization  of the Helmholtz equation.
When employing the optimal test norm, the `ideal' method yields the best approximation to the pair of the Helmholtz solution and its scaled gradient 
w.r.t.~the norm on $L_2(\Omega)\times L_2(\Omega)^d$ from the selected finite element trial space.
On convex polygons, the `practical', implementable method is shown to be pollution-free essentially whenever the order \new{$\tilde{p}$} of the finite element test space grows proportionally with \new{$\max(\log \kappa,p^2)$}, \new{with $p$ being the order at trial side.}
Numerical results also on other domains show a much better accuracy than for the Galerkin method.
\end{abstract}

\maketitle
\section{Introduction}
\subsection{Numerical approximation of the Helmholtz problem}
Standard Galerkin discretizations of the Helmholtz equation suffer from so-called pollution. For large wavenumbers $\kappa$, the Galerkin solution is not a quasi-best approximation from the selected finite element trial space. It is known that pollution can be avoided by choosing the polynomial degree $p$ sufficiently large dependent on $\kappa$. In the seminal work \cite{202.8} it has been shown that for a two- or three-dimensional bounded domain with analytic boundary, pure Robin boundary conditions, and a quasi-uniform mesh with mesh-size $h$, the Galerkin solution is a quasi-best approximation from the trial space w.r.t.~the norm $\|\cdot\|_{1,\kappa}:=\sqrt{\frac{1}{\kappa^2}|\cdot|_{H^1(\Omega)}^2+\|\cdot\|_{L_2(\Omega)}^2}$
when $\frac{h \kappa}{p}$ and $\frac{\log \kappa}{p}$ are sufficiently small. The same result holds true on (two-dimensional) polygons when the mesh is properly additionally refined towards the vertices \cite{202.8,71.5}. Notice that for a typical solution of the Helmholtz problem, $\frac{h \kappa}{p}$ being sufficiently small is already desirable from an approximation point of view.

Besides plain finite element Galerkin, numerous other numerical solution methods for the Helmholtz equation have been proposed.
Several of those are based on approximation properties of problem adapted systems of functions (e.g. \cite{138.296,243.53}).
Methods that are based on the approximation by piecewise polynomials, as we will consider, include (First Order) Least-Squares methods (\cite{182.35,38.72,20.186}, and Discontinuous Galerkin methods.
A subclass of the latter methods is generated by the Ultra-Weak Variational Formulation (UWVF) (\cite{35.9373,35.8656}), which despite its name is quite different from the method that is studied in this work.

\subsection{Ultra-weak first order system formulation, and optimal test norm}
We \new{are going to} write the Helmholtz equation as an ultra-weak first order variational system $\langle \mathbbm{u}, B_\kappa' \mathbbm{v}\rangle_U=q(\mathbbm{v})$ ($ \mathbbm{v} \in V_\mp$). Here \new{$U:=L_2(\Omega)\times L_2(\Omega)^d$, $V_\mp$ is a closed subspace of $H^1(\Omega)\times H(\divv;\Omega)$ defined by the incorporation of \new{(adjoint)} homogeneous boundary conditions (the symbol $\pm$ refers to either of the possible signs in the Robin boundary condition, and $\mp=-\pm$)}, 
$\mathbbm{u}=(\phi,\frac{1}{\kappa} \nabla \phi)$, where $\phi$ is the Helmholtz solution, $B_\kappa'$ is a partial differential operator of first order, \new{and $q \in V_\mp'$ is a functional defined in terms of the data of the Helmholtz problem.}
We show that for any $\kappa>0$ this formulation is well-posed in the sense that $B_\kappa$, i.e., the adjoint of $B_\kappa'$, is  a boundedly invertible operator
 from $U$ to the dual space $V_\mp'$.
 
 When both $U$ and $V_\mp$ are equipped with their canonical norms, the condition number of $B_\kappa$ increases with increasing $\kappa$. By replacing the canonical norm on $V_\mp$ by the so-called optimal test norm $\|\cdot\|_{V_\mp,\kappa}:=\|B'_\kappa\cdot\|_U$, and by equipping $V_\mp'$ with the resulting dual norm, $B_\kappa$ becomes an isometry. 
Consequently, given a finite dimensional subspace $U^\delta \subset U$, the `ideal' least-squares approximation $\mathbbm{u}^\delta:=\argmin_{\mathbbm{w}^\delta \in U^\delta}\|q-B_\kappa \mathbbm{w}^\delta\|_{V_\mp'}$ is the \emph{best} approximation from $U^\delta$ to $\mathbbm{u}$  w.r.t.~the norm on $U$.
Notice that $\|\mathbbm{u}\|_U=\|(\phi,\frac{1}{\kappa} \nabla \phi)\|_{L_2(\Omega)\times L_2(\Omega)^d}=\|\phi\|_{1,\kappa}$.

The use of the optimal test norm has been advocated in \new{\cite{64.155,45.44}}, see also \cite{35.8565}.
Presented in a somewhat different way, it can also be found in \cite{19.83}.

\subsection{`Practical' method}
The residual minimizer w.r.t.~the norm $\sup_{0 \neq \mathbbm{v}\in V_\mp} \frac{|\cdot(\mathbbm{v})|}{\|B_\kappa' \mathbbm{v}\|_{U}}$ on $V_\mp'$ is not computable.
By replacing the supremum over $\mathbbm{v}\in V_\mp$ by a supremum over $\mathbbm{v}^\delta\in V_\mp^\delta$ for some (sufficiently large) finite dimensional subspace $V_\mp^\delta \subset V_\mp$
we obtain an implementable `practical' method. Its solution, that we still denote with $\mathbbm{u}^\delta$, is the second component of the pair $(\mathbbm{v}^\delta,\mathbbm{u}^\delta) \in V_\pm^\delta \times U^\delta$ that solves 
\be \label{zadel}
\left\{
\begin{array}{lcll}
\langle B_\kappa'\mathbbm{v}^\delta, B_\kappa'\undertilde{\mathbbm{v}}^\delta\rangle_U +
\langle \mathbbm{u}^{\delta}, B_\kappa'\undertilde{\mathbbm{v}}^\delta\rangle_U 
& \!\!= \!\!& q(\undertilde{\mathbbm{v}}^\delta) & (\undertilde{\mathbbm{v}}^\delta \in V_{\mp}^\delta),\\
\langle B_\kappa'\mathbbm{v}^\delta,\undertilde{\mathbbm{u}}^\delta\rangle_U
& \!\!=\!\! & 0 & (\undertilde{\mathbbm{u}}^\delta \in U^\delta).
\end{array}
\right.
\ee
We refer to $V^\delta_\mp$ and $U^\delta$ as being the test and trial space.

\begin{remark} First order ultra-weak formulations of Helmholtz equations were studied earlier in \cite{64.155,75.63} in a DPG context.
In that setting, the solution is a quadruple that besides the `field variables' $\phi$ and its gradient, contains two trace variables that have as domain the mesh skeleton.
These trace variables are measured in intrinsically stronger norms, which has the consequence that, in order to guarantee a certain convergence rate for the field variables, stronger regularity conditions are needed. The analysis in \cite{64.155,75.63} was restricted to the `ideal' method in which the residual is minimized in \new{the} non-computable dual norm.
\end{remark}

With the inf-sup constant
$\gamma^\delta_\kappa=\gamma^\delta_\kappa(U^\delta,V_\mp^\delta):= \inf_{0 \neq \undertilde{\mathbbm{u}}^\delta \in U^\delta} \sup_{0 \neq \undertilde{\mathbbm{v}}^\delta \in V_{\mp}^\delta} \frac{|
\langle \undertilde{\mathbbm{u}}^{\delta}, B_\kappa'\undertilde{\mathbbm{v}}^\delta\rangle_U|}{\|\undertilde{\mathbbm{u}}^\delta\|_U\|B_\kappa' \undertilde{\mathbbm{v}}^\delta\|_{U}}$, it will be shown that
\be \label{poll}
\sup_{\mathbbm{u} \in U\setminus U^\delta}
\frac{
\|\mathbbm{u}-\mathbbm{u}^\delta\|_U
}{
\inf_{\undertilde{\mathbbm{u}}^\delta \in U^\delta} \|\mathbbm{u}-\undertilde{\mathbbm{u}}^\delta\|_U
}= \frac{1}{\gamma^\delta_\kappa}.
\ee
In other words, the error in $\mathbbm{u}^\delta$ is at most a factor $1/\gamma^\delta_\kappa$ larger than the error in the best approximation from $U^\delta$, and for some $\mathbbm{u}$, it \emph{will} be a factor $1/\gamma^\delta_\kappa$ larger. We will therefore call $1/\gamma^\delta_\kappa$ the pollution factor of the method.
The constant $\gamma^\delta_\kappa$ can be computed as the square root of the smallest eigenvalue of a generalized eigenvalue problem in terms of $\kappa$, $U^\delta$ and $V^\delta_{\new{\mp}}$.

The least squares approximation $\mathbbm{u}^{\delta}$ can be called to be quasi-best when $\delta$ is from a collection $\Delta$ for which 
 $\overline{\cup_{\delta \in \Delta} U^\delta}=U$, and $\inf_{\delta \in \Delta,\,\kappa>0} \gamma^\delta_\kappa(U^\delta,V_\mp^\delta(\kappa))>0$. Notice that here we allow $V_\mp^\delta$ to (mildly) depend on $\kappa$.
 To demonstrate this uniform `inf-sup stability', we give an alternative expression for $\gamma^\delta_\kappa$ in terms of approximability from $V_\mp^\delta$ of the solution of an adjoint first order Helmholtz problem with a forcing term from $U^\delta$.

 Using this expression, for pure Robin boundary conditions, and $U^\delta$ being the space of (dis)continuous piecewise polynomials of degree $p$ w.r.t.~a quasi-uniform triangulation with mesh-size $h_\delta$ of a convex polygon $\Omega \subset \R^2$, we show uniform inf-sup stability for $V_\pm^\delta$ being the product of the continuous piecewise polynomials of degree $\tilde{p}$ and the Raviart-Thomas finite elements of order $\tilde{p}$ w.r.t.~a quasi-uniform triangulation with mesh-size $\tilde{h}_\delta$, with some additional refinements near the vertices, under the condition that \new{$\max(\log \kappa,p^2)/\tilde{p}$} is sufficiently small, and $\tilde{h}_\delta \lesssim h_\delta^{1+\eps}$ for arbitrary $\eps>0$. Our analysis builds on results from \cite{202.8} for the standard Galerkin discretization.
 
\new{The (non-exhaustive) numerical tests that we have performed so far show already good results for $U^\delta$ and $V^\delta_{\new{\mp}}$ finite element spaces w.r.t.~the same triangulation with orders $p$ and $\tilde{p}=p+2$ respectively.}

 \subsection{A posteriori error estimator}
 We will show that $\mathbbm{u}^\delta+ B_\kappa'\mathbbm{v}^\delta$ is a better approximation to $\mathbbm{u}$ than $\mathbbm{u}^\delta$, and that $\|B_\kappa'\mathbbm{v}^\delta\|_U \leq \|\mathbbm{u}-\mathbbm{u}^\delta\|_U$.
 In experiments, we observe that what we will call the `boosted approximation' $\mathbbm{u}^\delta+ B_\kappa'\mathbbm{v}^\delta$, converges to $\mathbbm{u}$ with a better rate than $\mathbbm{u}^\delta$ does, so that the a posteriori estimator $\|B_\kappa'\mathbbm{v}^\delta\|_U$ for $\|\mathbbm{u}-\mathbbm{u}^\delta\|$ is asymptotically exact. We will use local norms of $B_\kappa'\mathbbm{v}^\delta$ as error indicators to drive an adaptive solution method.
 
  \subsection{Comparison with standard Galerkin}
 To compare the results of our FOSLS discretization with that of the standard Galerkin discretization, we define the pollution factor $1/\gamma_{\kappa,{\rm Gal}}^\delta$
 for the latter analogously to \eqref{poll} but now w.r.t.~the $\|\cdot\|_{1,\kappa}$-norm, and with $U$ and $U^\delta$ replaced by $H^1(\Omega)$ and a finite element space $X^\delta \subset H^1(\Omega)$, respectively.
 This constant $\gamma_{\kappa,{\rm Gal}}^\delta$ is, however, not computable, and we approximate it by a computable quantity
 $\hat{\gamma}_{\kappa,{\rm Gal}}^\delta$ for which $\gamma_{\kappa,{\rm Gal}}^\delta \leq \hat{\gamma}_{\kappa,{\rm Gal}}^\delta \leq C \gamma_{\kappa,{\rm Gal}}^\delta$ for some constant $C>0$, which in numerical experiments turns out to be very close to $1$.
 Notice that $1/\hat{\gamma}_{\kappa,{\rm Gal}}^\delta$ is a lower bound for the pollution factor meaning that there exist solutions of the Helmholtz problem for which the $\|\cdot\|_{1,\kappa}$-error in the Galerkin approximation is at least a factor $1/\hat{\gamma}_{\kappa,{\rm Gal}}^\delta$ larger than this error in the best approximation from $X^\delta$.
 
 Numerical results for uniform triangulations of the unit square confirm that the Galerkin method is only pollution-free under the conditions of $\frac{h \kappa}{p}$ and $\frac{\log \kappa}{p}$ being sufficiently small, whereas the ultra-weak FOSLS discretization is pollution free whenever the \new{order} $\tilde{p}$ at the test side satisfies  the condition that  \new{$\max(\log \kappa,p^2)/\tilde{p}$} is sufficiently small.
 
 For $\kappa \approx 100$, mixed boundary conditions, several $p$ and $\tilde{p}=p+2$, uniform and locally refined meshes equal at test- and trial-side, and several domains including `trapping domains', tests show a much better accuracy of the ultra-weak FOSLS method in comparison to the standard Galerkin method. 
 
 On the other hand, it is fair to say that computing the ultra-weak FOSLS solution requires solving the larger saddle-point problem \eqref{zadel}.
 An iterative solution of this problem, however, only requires a good preconditioner for the Hermitian positive definite left upper block, \new{which might 
be easier to construct than such a preconditioner for the indefinite matrix that results from the standard Galerkin discretization.} This will be study of future work.
 
\subsection{Organization}
In Sect.~\ref{sec:helm} we write the Helmholtz equation as a first order system.
By applying integration-by-parts, in Sect.~\ref{sec:ultra-weak} we derive an ultra-weak formulation, and prove its well-posedness
as a mapping from $U=L_2(\Omega) \times L_2(\Omega)^d$ to the dual of a space $V_\mp$.
After equipping the latter space with the optimal test norm, we consider `ideal' and `practical' least-squares discretizations in Sect.~\ref{sec:ultra-weakFOSLS}.
The latter discretizations yield quasi-best solutions under a uniform inf-sup condition. Verification of this condition is the topic of Sect.~\ref{sec:infsupstab}.
In Sect.~\ref{sec:assessmentGalerkin} we introduce a computable pollution factor of the common Galerkin discretization with the purpose to compare this method with the FOSLS method.
Sect.~\ref{sec:boosted} deals with a boosted FOSLS method and a posteriori error estimation. Numerical results are presented in Sect.~\ref{sec:numerics}.
Finally, conclusions are formulated in Section~\ref{sec:conclusions}.

\subsection{Notations}
For normed linear spaces $E$ and $F$, by $\cL(E, F)$ we will denote the normed linear
space of bounded linear mappings $E \rightarrow F$, and by $\Lis(E, F)$ its subset of boundedly invertible linear mappings $E \rightarrow F$. We write $E \hookrightarrow F$ to denote that $E$ is
continuously embedded into $F$. Since we consider linear spaces over $\C$, for a normed linear space $E$ its dual $E'$ is the normed linear space of anti-linear functionals.

By $C \lesssim D$ we will mean that $C$ can be bounded by a multiple of $D$, \new{unless explicitly stated otherwise} \emph{independently} of parameters which $C$ and $D$ may depend on, as the wave number $\kappa$, the discretisation index $\delta$, \new{or the orders $p$ and $\tilde{p}$ at trial- and test-side.}
Obviously, $C \gtrsim D$ is defined as $D \lesssim C$, and $C\eqsim D$ as $C\lesssim D$ and $C \gtrsim D$.
\new{The aforementioned multiple may depend on the space dimension $d$, and on the shape regularity of the finite element mesh.}

\section{Helmholtz equation} \label{sec:helm}
\subsection{Second order formulation}
Let $\Omega \subset \R^d$ be a bounded Lipschitz domain with boundary $\partial\Omega$ decomposed into $\Gamma_D \,\dot{\cup}\, \Gamma_N \,\dot{\cup}\,\Gamma_R$, where $|\Gamma_R|>0$, and \new{for some arbitrary, but fixed $\kappa_0>0$, let $\kappa\geq \kappa_0$.} 
Given $f \in H^1_{0,\Gamma_D}(\Omega)'$, $g_D \in H^{\frac12}(\Gamma_D)$,
and $g \in H^{-\frac12}(\Gamma_N \cup \Gamma_R)$ (=$H_{00}^{\frac12}(\Gamma_N \cup \Gamma_R)'$)%
,\footnote{\new{For a measurable $\Gamma \subset \partial\Omega$, $H^1_{0,\Gamma}(\Omega):=\{v \in H^1(\Omega)\colon v|_{\Gamma}=0\}$.
$H^{\frac12}(\Gamma)$ is the space of restrictions of $H^1(\Omega)$-functions to $\Gamma$ with norm $\|v\|_{H^{\frac12}(\Gamma)}:=\inf\{\|w\|_{H^1(\Omega)}\colon w|_{\Gamma}=v,\,w \in H^1(\Omega)\}$. The definition of $H_{00}^{\frac12}(\Gamma)$ is similar with $H^1(\Omega)$ replaced by its subspace $H_{0,\partial\Omega \setminus \Gamma}^1(\Omega)$.}}
we consider the Helmholtz equation with (mixed) Dirichlet, Neumann and/or Robin boundary conditions of finding $\phi\in H^1(\Omega)$ that satisfies
\be \label{eq:helmholtz}
\begin{aligned}
-\Delta \phi-\kappa^2 \phi &= \kappa^2 f & &\text{on } \Omega,\\
\phi & = \kappa g_D \quad & &\text{on } \Gamma_D,\\
\tfrac{\partial \phi}{\partial \vec{n}} & = \kappa^2 g \quad & &\text{on } \Gamma_N,\\
\tfrac{\partial \phi}{\partial \vec{n}}\pm i \kappa \phi&=\kappa^2 g \quad  & &\text{on } \Gamma_R,
\end{aligned}
\ee
where `$\pm$' means either of the valid options `$+$' or `$-$'.

\begin{remark} The scalings at the right-hand side with factors $\kappa$ or $\kappa^2$ are harmless because $\kappa>0$, and are made for convenience.
\end{remark}

In variational form  \eqref{eq:helmholtz} reads as finding
$$
\phi\in H^1_{\kappa g_D,\Gamma_D}(\Omega):=\{\breve{\phi} \in H^1(\Omega)\colon \breve{\phi}|_{\Gamma_D}=\kappa g_D\},
$$
such that 
\be \label{2nd}
\begin{split}
(L_\kappa \phi)(\eta):=&\int_\Omega \tfrac{1}{\kappa^2} \nabla \phi  \cdot \nabla \overline{\eta}-\phi \overline{\eta}\,dx \pm  \tfrac{i}{\kappa}  \int_{\Gamma_R} \phi \overline{\eta}\,ds\\
&=f(\eta)+ \int_{\Gamma_N \cup \Gamma_R}  g \overline{\eta}\,ds \quad (\eta \in H_{0,\Gamma_D}^1(\Omega)).
\end{split}
\ee
It is known, see e.g.~\cite[Ch.~35]{70.98}, that
\be \label{wellposedness}
L_\kappa \in \Lis(H_{0,\Gamma_D}^1(\Omega),H_{0,\Gamma_D}^1(\Omega)'),
\ee
albeit with a condition number that tends to $\infty$ when $\kappa \uparrow \infty$.

\begin{remark}[No Robin boundary] The well-posedness \eqref{wellposedness} holds also true 
when $|\Gamma_R|=0$ and $\kappa^2$ is not an eigenvalue of the Laplace operator with (homogeneous) mixed Dirichlet/Neumann boundary conditions.
In those cases also the forthcoming Theorem~\ref{thm:main} is valid.
\end{remark}

\subsection{First order formulation}
For some $f_1 \in L_2(\Omega)$ and $\vec{f}_2 \in L_2(\Omega)^d$, we decompose $f \in H^1_{0,\Gamma_D}(\Omega)'$ as
$$
f(\eta)=\int_\Omega f_1 \overline{\eta} + \tfrac{1}{\kappa} \vec{f}_2 \cdot \nabla \overline{\eta}\,dx \quad (\eta \in H^1_{0,\Gamma_D}(\Omega)),
$$
which decomposition exists, although non-uniquely, by Riesz' representation theorem.
Then setting $\vec{u}=\tfrac{1}{\kappa} \nabla \phi - \vec{f}_2$ in \eqref{2nd}, we arrive at the \emph{first order system}
\be \label{1storder}
\begin{aligned}
-\tfrac{1}{\kappa} \divvv \vec{u}- \phi&= f_1 & &\text{on } \Omega,\\
\tfrac{1}{\kappa} \nabla \phi- \vec{u}&=\vec{f}_2 & &\text{on } \Omega,\\
\phi & = \kappa g_D \quad  & &\text{on } \Gamma_D,\\
\vec{u} \cdot \vec{n} & = \kappa g & &\text{on } \Gamma_N,\\
\vec{u} \cdot \vec{n} \pm i \phi&=\kappa g \quad & &\text{on } \Gamma_R.
\end{aligned}
\ee

\section{Ultra-weak first order formulation} \label{sec:ultra-weak}
\subsection{Definition and well-posedness}
Assuming sufficiently smooth data, by testing the first equation in \eqref{1storder} with $\eta$ and the second one with $\vec{v}$, for smooth $\eta$ and $\vec{v}$ with 
$\eta=0$ on $\Gamma_D$, $\vec{v}\cdot \vec{n}=0$ on $\Gamma_N$, and
$\vec{v}\cdot\vec{n} \mp i \eta =0$ on $\Gamma_R$, and by applying integration-by-parts and substituting the boundary conditions from \eqref{1storder} we arrive at the ultra-weak variational formulation
\be \label{eq:ultra-weak}
\begin{split}
\big(B_\kappa&(\phi,\vec{u})\big)(\eta,\vec{v}):=
\int_\Omega  \tfrac{1}{\kappa} \vec{u}\cdot\nabla \overline{\eta} -  \phi \overline{\eta}- \tfrac{1}{\kappa} \phi \divvv \overline{\vec{v}}- \vec{u}\cdot \overline{\vec{v}}\,dx\\
&=\int_\Omega f_1 \overline{\eta}+\vec{f}_2 \cdot \overline{\vec{v}}\,dx
-\int_{\Gamma_D} g_D \overline{\vec{v}}\cdot\vec{n}  \,ds+ \int_{\Gamma_N \cup \Gamma_R} g \overline{\eta} \,ds=:q(\eta,\vec{v}).
\end{split}
\ee

We set
\begin{align*}
&U:=L_2(\Omega) \times L_2(\Omega)^d,
\intertext{and}
&V_{\mp}:=\Big\{(\eta,\vec{v}) \in H_{0,\Gamma_D}^1(\Omega) \times H(\divv;\Omega)\colon \\
&\hspace*{9em}\int_{\partial\Omega} \vec{v}\cdot\vec{n} \,\overline{\psi} \,ds \mp i \int_{\Gamma_R} \eta \overline{\psi} \,ds=0 \quad (\psi\in H_{0,\Gamma_D}^1(\Omega))\Big\},
\end{align*}%
both being Hilbertian spaces equipped with their canonical norms $\|\cdot\|_{L_2(\Omega) \times L_2(\Omega)^d}$ and $\|\cdot\|_{H^1(\Omega) \times H(\divv;\Omega)}$.
We will use that each $\eta \in H_{0,\Gamma_D}^1(\Omega)$ can be completed to a pair $(\eta,\vec{v}) \in V_{\mp}$:


\begin{lemma} \label{lem1} For each $\eta \in H^1_{0,\Gamma_D}(\Omega)$, there exists a $\vec{v} \in H(\divv;\Omega)$ with $(\eta,\vec{v}) \in V_{\mp}$.
\end{lemma}

\begin{proof}
Given  $\eta \in H^1_{0,\Gamma_D}(\Omega)$, take $\vec{v}:=\nabla u$ where $u \in H^1_{0,\Gamma_D}(\Omega)$ solves
$$
\int_\Omega \nabla u\cdot \nabla \overline{\psi}+u  \overline{\psi} \,dx=\pm i \int_{\Gamma_R} \eta \overline{\psi}\,ds\quad (\psi  \in H^1_{0,\Gamma_D}(\Omega)).
$$
Then $\divv \vec{v}=u$
and $\pm i \int_{\Gamma_R} \eta \overline{\psi}\,ds=\int_\Omega \vec{v} \cdot \nabla \overline{\psi}+\divv \vec{v} \, \overline{\psi} \,dx=\int_{\partial \Omega} \vec{v}\cdot \vec{n}\,\overline{\psi}\,ds$ ($\psi  \in H^1_{0,\Gamma_D}(\Omega)$), i.e., $(\eta,\vec{v}) \in V_{\mp}$.
\end{proof}

The next theorem shows that for $q \in V_{\mp}'$, and for any fixed $\kappa>0$, finding $(\phi,\vec{u}) \in U$ such that \eqref{eq:ultra-weak} holds for all $(\eta,\vec{v}) \in V_{\mp}$ is a \emph{well-posed, consistent ultra-weak first order formulation of the Helmholtz equation \eqref{eq:helmholtz}}.
Note that in this formulation \emph{all boundary conditions are natural}.

\begin{theorem} \label{thm:main} For any fixed $\kappa>0$, it holds that $B_\kappa \in \Lis(U,V_{\mp}')$.
\end{theorem}

\begin{proof} Membership of $B_\kappa \in \cL(U,V_{\mp}')$ is immediate.
In Steps (i) and (ii) we show that $B_\kappa$ is injective and surjective which completes the proof by an application of the open mapping theorem.

(i) Let $(\phi,\vec{u}) \in \ker B_\kappa$. Using that $\big(B_\kappa(\phi,\vec{u})\big)(\eta,\vec{v})$ 
vanishes in particular for smooth $(\eta,\vec{v})$ with support within $\Omega$, it follows that in distributional sense
$$
\divvv \vec{u}+\kappa \phi=0,\quad
\nabla\phi-\kappa \vec{u}=0,
$$
and so in particular $\phi\in H^1(\Omega)$ and $\vec{u} \in H(\divv;\Omega)$.
This allows us to apply integration-by-parts, and to conclude that
\be \label{8}
\int_{\partial\Omega} \vec{u}\cdot\vec{n} \,\overline{\eta}-\phi \overline{\vec{v}}\cdot\vec{n}\,ds=0 \quad((\eta,\vec{v})\in V_{\mp}).
\ee


Taking $(\eta,\vec{v}) \in \{0\}\times H_{0,\Gamma_N \cup \Gamma_R}(\divv;\Omega)$ in \eqref{8}\footnote{\new{$H_{0,\Gamma_N \cup \Gamma_R}(\divv;\Omega):=\{\vec{v}\in H(\divv;\Omega)\colon \vec{v} \cdot\vec{n}|_{\Gamma_N \cup \Gamma_R}=0 \text{ in } H^{-\frac12}(\Gamma_N \cup \Gamma_R)\}$.}} we infer that
\be \label{e0}
\phi \in H^1_{0,\Gamma_D}(\Omega).
\ee
(Indeed for $|\Gamma_D|>0$ and $\phi \in H^1(\Omega)$, let $u$ solve $-\Delta u + u=0$, $\tfrac{\partial u}{\partial \vec{n}}=\phi$ on $\Gamma_D$, $\tfrac{\partial u}{\partial \vec{n}}=0$ on $\Gamma_N \cup \Gamma_R$. Then $\vec{v}:=\nabla u \in H_{0,\Gamma_N \cup \Gamma_R}(\divv;\Omega)$, and $\int_{\Gamma_D} |\phi|^2\,dx=\int_{\Gamma_D} \phi \,\overline{\vec{v}}\cdot \vec{n}\,dx=0$ shows \eqref{e0}.)

Given $\eta \in H^1_{0,\Gamma_D}(\Omega)$, let $\vec{v} \in H(\divv;\Omega)$ be such that $(\eta,\vec{v}) \in V_{\mp}$, see Lemma~\ref{lem1}.
Then, thanks to  $\phi \in H^1_{0,\Gamma_D}(\Omega)$, the definition of  $V_{\mp}$ and \eqref{8} show that
$$
\int_{\partial\Omega} \vec{u}\cdot \vec{n} \,\overline{\eta}\,ds \pm i \int_{\Gamma_R} \phi \overline{\eta}\,ds=
\int_{\partial\Omega} \vec{u}\cdot \vec{n} \,\overline{\eta}\,ds -\int_{\partial\Omega} \phi \overline{\vec{v}}\cdot\vec{n} \,ds=0.
$$
From $\nabla\phi=\kappa \vec{u}$ and $\divvv \vec{u}+\kappa \phi=0$, we infer that for 
$\eta \in H_{0,\Gamma_D}^1(\Omega)$,
\begin{align*}
(L_\kappa \phi)(\eta) &= \int_\Omega \tfrac{1}{\kappa^2} \nabla \phi\cdot \nabla \overline{\eta}- \phi \overline{\eta} \,dx \pm \tfrac{i}{\kappa} \int_{\Gamma_R} \phi \overline{\eta}\,ds
\\ &=
\int_\Omega \tfrac{1}{\kappa} \vec{u}\cdot \nabla \overline{\eta}- \phi \overline{\eta} \,dx- \tfrac{1}{\kappa} \int_{\partial\Omega} \vec{u}\cdot\vec{n}\, \overline{\eta}\,ds
\\ & = - \tfrac{1}{\kappa} \int_\Omega (\divvv \vec{u}+ \kappa \phi) \overline{\eta}\,dx=0.
\end{align*}
From $L_\kappa \in \cL(H_{0,\Gamma_D}^1(\Omega),H_{0,\Gamma_D}^1(\Omega)')$ being injective, we conclude that $\phi=0$ and so $\vec{u}=0$, and thus that $B_\kappa$ is injective.

(ii) To show surjectivity of $B_\kappa$, let $q \in V_{\mp}'$. Riesz' representation theorem shows that there exists a pair $(\zeta,\vec{w}) \in V_{\mp}$ such that
$$
q(\eta,\vec{v})=\int_\Omega \vec{w}\cdot \overline{\vec{v}}+\divvv \vec{w} \,\divvv \overline{\vec{v}}+\zeta \overline{\eta}+\nabla \zeta \cdot \nabla \overline{\eta} \,dx
\quad( (\eta,\vec{v}) \in V_{\mp}).
$$
Setting $\undertilde{\vec{u}}\new{:=}\vec{u}+\vec{w}$, $\undertilde{\phi}\new{:=}\phi +\kappa \divvv \vec{w}$,
the equation $B_\kappa (\phi,\vec{u})=q$ is equivalent to 
\be \label{3}
\begin{split}
\int_\Omega \tfrac{1}{\kappa} \undertilde{\vec{u}} \cdot \nabla \overline{\eta}&- \undertilde{\phi} \overline{\eta}-  \tfrac{1}{\kappa} \undertilde{\phi} \,\divvv \overline{\vec{v}}-  \undertilde{\vec{u}}  \cdot 
\overline{\vec{v}}\,dx=\\
& \int_\Omega 
(\zeta-\kappa \divvv \vec{w})\overline{\eta}+(\nabla \zeta+\tfrac{\vec{w}}{\kappa})\cdot \nabla \overline{\eta}\,dx  =:\undertilde{q}(\eta)\quad((\eta,\vec{v}) \in V_{\mp}),
\end{split}
\ee
where $\undertilde{q} \in H_{0,\Gamma_D}^1(\Omega)'$.

Given  $\undertilde{\phi} \in H_{0,\Gamma_D}^1(\Omega)$, let $\undertilde{\vec{u}}:=\frac{1}{\kappa} \nabla \undertilde{\phi}$. 
Then \eqref{3} reads as
\be \label{4}
\int_\Omega \tfrac{1}{\kappa^2} \nabla \undertilde{\phi}  \cdot \nabla \overline{\eta}-  \undertilde{\phi} \overline{\eta}- 
\tfrac{1}{\kappa} ( \undertilde{\phi} \divvv \overline{\vec{v}}+\nabla \undertilde{\phi} \cdot \overline{\vec{v}})\,dx= \undertilde{q}(\eta)\quad((\eta,\vec{v}) \in V_{\mp}).
\ee
Using that $\int_\Omega \undertilde{\phi}  \divvv \overline{\vec{v}}+\nabla \undertilde{\phi} \cdot \overline{\vec{v}}\,dx=
\int_{\partial\Omega}  \undertilde{\phi} \overline{\vec{v}} \cdot \vec{n}\,dx= \mp i \int_{\Gamma_R} \undertilde{\phi} \overline{\eta}\,ds$
by definition of $V_{\mp}$, it follows that \eqref{4} is equivalent to $L_\kappa \undertilde{\phi}= \undertilde{q}$.
Knowing that $L_\kappa \in \cL(H_{0,\Gamma_D}^1(\Omega),H_{0,\Gamma_D}^1(\Omega)')$ is surjective, a solution of the latter equation exists, which completes the proof.
\end{proof}

\begin{remark}(Alternative \new{definition of a ultra-weak formulation})  The boundary condition $\vec{v}\cdot\vec{n} \mp i \eta =0$ on $\Gamma_R$, which is incorporated in the definition of $V_{\mp}$, will pose some challenges for the analysis. Considering $\partial\Omega=\Gamma_R$, alternatively we expect that well-posedness similarly to Theorem~\ref{thm:main} can also be shown for $V_{\mp}$ and $U$ being replaced by $H^1(\Omega) \times H(\divv;\Omega)$ and $L_2(\Omega) \times L_2(\Omega)^d \times H^{\frac12}(\partial\Omega)$, respectively, the latter factor being the space for an additional variable \new{$\tilde{\phi}$ that represents} the trace $\phi|_{\partial\Omega}$, and
\new{sesquilinear form $(\tilde{B}_\kappa(\phi,\vec{u},\tilde{\phi}))(\eta,\vec{v}):=(B_\kappa(\phi,\vec{u}))(\eta,\vec{v})+\int_{\partial\Omega} \frac{1}{\kappa}\tilde{\phi}(\overline{\vec{v}}\cdot\vec{n}\pm i \overline{\eta})\,ds$}.
In this case, however, for non-smooth solutions best approximation errors in $L_2(\Omega) \times L_2(\Omega)^d \times H^{\frac12}(\partial\Omega)$ will generally be dominated by the error in this additional trace variable \new{$\tilde{\phi}$}. We therefore do not investigate this option.
\end{remark}

\subsection{Adjoint problem} 
From $U \simeq U'$, one has that $B_{\kappa} \in \Lis(U,V_{\mp}')$ is equivalent to 
$$
B'_{\kappa}\colon (\eta,\vec{v}) \mapsto (- \tfrac{1}{\kappa} \divvv \vec{v}- \eta, \tfrac{1}{\kappa} \nabla \eta- \vec{v}) \in \Lis(V_{\mp},U).
$$

The following lemma relates an adjoint first order system to an adjoint second order problem, both with homogeneous boundary conditions.
The approximability of the solution of an adjoint second order problem by finite element functions was the key to analyze standard finite element Galerkin discretizations of the Helmholtz problem.
Later, in Sect.~\ref{sec:practical}, we will use results from \cite{202.8} about this approximability to analyze the approximability of an adjoint first order system, which in turn will be the key to demonstrate quasi-optimality of our FOSLS discretization.

\begin{lemma} \label{lem:easy} For $h_1 \in L_2(\Omega)$ and $\vec{h}_2 \in H_{0,\Gamma_N \cup \Gamma_R}(\divv;\Omega)$ it holds that $(\eta,\vec{v}) =
(\eta,\frac{1}{\kappa}\nabla \eta-\vec{h}_2)\in V_{\mp}$ solves
$$
B'_{\kappa} (\eta,\vec{v})=(h_1,\vec{h}_2)
$$
if and only if $\eta \in H^1_{0,\Gamma_D}(\Omega)$ solves
$$
(L'_\kappa \eta)(\zeta):=\int_\Omega \tfrac{1}{\kappa^2} \nabla \eta  \cdot \nabla \overline{\zeta}-\eta \overline{\zeta}\,dx \mp \tfrac{i}{\kappa}
 \int_{\Gamma_R} \eta \overline{\zeta}\,ds=\int_\Omega ( h_1 -\tfrac{1}{\kappa} \divvv \vec{h}_2)  \overline{\zeta} \,dx.
$$
$(\zeta \in H_{0,\Gamma_D}^1(\Omega))$.
\end{lemma}

\begin{proof} For completeness we provide the easy proof. If $(\eta,\vec{v}) \in V_{\mp}$ solves the first order system, then
\begin{align*}
\int_\Omega \tfrac{1}{\kappa^2} \nabla \eta  \cdot \nabla \overline{\zeta}-\eta \overline{\zeta}\,dx \mp \tfrac{i}{\kappa}
\int_{\Gamma_R} \eta \overline{\zeta}\,ds=\int_\Omega \tfrac{1}{\kappa} (\vec{h}_2+\vec{v})  \cdot \nabla \overline{\zeta}-\eta \overline{\zeta}\,dx \mp 
\tfrac{i}{\kappa}\int_{\Gamma_R} \eta \overline{\zeta}\,ds\\
=\int_\Omega (-\tfrac{1}{\kappa} \divvv (\vec{h}_2+\vec{v}) -\eta)  \overline{\zeta}\,dx=\int_\Omega (h_1 -\tfrac{1}{\kappa} \divvv \vec{h}_2)  \overline{\zeta} \,dx.
\end{align*}

Conversely, let $\eta \in H^1_{0,\Gamma_D}(\Omega)$ satisfy the variational formulation of the second order equation. Then setting $\vec{v}=\frac{1}{\kappa}\nabla \eta-\vec{h}_2$, we have
$$
\int_\Omega \tfrac{1}{\kappa} \vec{v}  \cdot \nabla \overline{\zeta}-\eta \overline{\zeta}\,dx \mp \tfrac{i}{\kappa} \int_{\Gamma_R} \eta \overline{\zeta}\,ds=\int_\Omega h_1   \overline{\zeta} \,dx \quad (\zeta \in H_{0,\Gamma_D}^1(\Omega)).
$$
It shows that $\vec{v}  \in H(\divv;\Omega)$, and from a subsequent integration-by-parts one infers $- \tfrac{1}{\kappa} \divvv \vec{v}- \eta=h_1$, and
$\int_{\partial\Omega} \vec{v}\cdot\vec{n} \,\overline{\zeta} \,ds \mp i \int_{\Gamma_R} \eta \overline{\zeta} \,ds=0$ $(\zeta \in H_{0,\Gamma_D}^1(\Omega))$.
\end{proof}

Although in this work we will focus on the ultra-weak first order formulation, finally in this section for completeness in the following remark it is shown that also the first order system \eqref{1storder} is well-posed w.r.t.~a proper choice of function spaces.

\begin{remark}[Well-posedness of the `mild' first order system \eqref{1storder}] \label{rem1} Since bounded invertibility of $B'_{\kappa}$ is equally well valid when $V_{\mp}$ is replaced by $V_{\pm}$, the system \eqref{1storder} with homogeneous boundary data $g_D=0$ and $g=0$ corresponds to a mapping in $\Lis(V_{\pm},U)$. 

Since furthermore  $B'_{\kappa}\in \cL(H^1(\Omega) \times H(\divv;\Omega),U)$, 
\begin{align*}
T\colon(\phi,\vec{u})\mapsto \Big(\phi|_{\Gamma_D},\psi \mapsto \int_{\partial\Omega} &\vec{u}\cdot\vec{n} \, \overline{\psi}\,ds \pm i \int_{\Gamma_R} \phi \overline{\psi}\,ds\Big) \\
&\in \cL\big(H^1(\Omega) \times H(\divv;\Omega), H^{\frac{1}{2}}(\Gamma_D) \times H^{-\frac{1}{2}}(\Gamma_N \cup \Gamma_R)\big)
\end{align*}
is surjective, and 
$V_{\pm}=\ker T$, from \cite[Lemma 2.7]{75.28} it follows that the system \eqref{1storder} for general boundary data corresponds to a mapping in
$$
\Lis(H^1(\Omega) \times H(\divv;\Omega), U \times 
H^{\frac{1}{2}}(\Gamma_D) \times H^{-\frac{1}{2}}(\Gamma_N \cup \Gamma_R)).
$$

To see the aforementioned surjectivity, let $(\eta,f) \in  H^{\frac{1}{2}}(\Gamma_D) \times H^{-\frac{1}{2}}(\Gamma_N \cup \Gamma_R)$. There exists a $\phi \in H^1(\Omega)$ with $\phi|_{\Gamma_D} = \eta$. With $\tilde{f} \in H^{-\frac{1}{2}}(\Gamma_N \cup \Gamma_R)$ defined by
$\tilde{f}(\psi)=f(\psi) \mp i \int_{\Gamma_R} \phi \overline{\psi}\,ds$, as in Lemma~\ref{lem1} let $z \in H^1_{0,\Gamma_D}(\Omega)$ solve
$$
\int_\Omega \nabla z \cdot \nabla \overline{\psi}+z  \overline{\psi} \,dx=\tilde{f}(\psi)\quad (\psi  \in H^1_{0,\Gamma_D}(\Omega)).
$$
Then $\vec{u}:=\nabla z$ satisfies $\int_{\partial \Omega} \vec{u}\cdot \vec{n}\,\overline{\psi}\,ds=\tilde{f}(\psi)$  ($\psi  \in H^1_{0,\Gamma_D}(\Omega)$).
\end{remark}

\
\section{Ultra-weak FOSLS} \label{sec:ultra-weakFOSLS}
We discretize the ultra-weak first order system $B_\kappa(\phi,\vec{u})=q$ using a least-squares or minimal residual discretisation. Recalling that $B_\kappa\in \Lis(U,V_{\mp}')$, when doing so we will equip $V_{\mp}$ with the so-called optimal test norm, which has the property that residual minimization in the resulting dual norm on $V_{\mp}'$
means error minimization in the canonical norm on $U=L_2(\Omega)\times L_2(\Omega)^d$.

\subsection{Optimal test norm}
In the following often we use the shorthand notations $\mathbbm{u}$, $\mathbbm{v}$ etc. to denote 
variables from $U$ and $V_\mp$ instead of writing them as pairs.
Using that $U \simeq U'$, we equip $V_{\mp}$ with the \emph{optimal test norm}
\be \label{eq:optimal test norm}
\|\mathbbm{v}\|_{V_{\mp,\kappa}}:= \|B'_\kappa \mathbbm{v}\|_{U},
\ee
\new{and corresponding inner product $\langle\,,\,\rangle_{V_{\mp,\kappa}}$,} 
and so $V_\mp'$ with the associated dual norm $\|\cdot\|_{V_{\mp,\kappa}'}$, \new{and corresponding inner product $\langle\,,\,\rangle_{V_{\mp,\kappa}'}$.}
Then 
$$
\|B_\kappa \mathbbm{u} \|_{V_{\mp,\kappa}'}=
\sup_{0 \neq \mathbbm{v}\in V_{\mp}} \frac{|(B_\kappa\mathbbm{u})(\mathbbm{v})|}{\|B'_\kappa \mathbbm{v}\|_U}
=
\sup_{0 \neq \mathbbm{v}\in V_{\mp}} \frac{|\langle \mathbbm{u},B'_\kappa\mathbbm{v}\rangle_U|}{\|B'_\kappa \mathbbm{v}\|_U}= \|\mathbbm{u}\|_U,
$$
i.e., $B_\kappa \in \Lis(U,V_{\mp}')	$ is an \emph{isometry}.

Consequently, for any $q \in V_{\mp}'$, and any closed, e.g.~finite dimensional subspace $\{0\} \subsetneq U^\delta \subset U$, the least squares solution
\be \label{ls}
\mathbbm{u}^\delta:=\argmin_{\mathbbm{w}^\delta \in U^\delta} \|q- B_\kappa \mathbbm{w}^\delta\|_{V_{\mp,\kappa}'}
\ee
is the \emph{best} approximation from $U^\delta$ w.r.t. $\|\cdot\|_U=\|\cdot\|_{L_2(\Omega) \times L_2(\Omega)^d}$ to the solution $\mathbbm{u}=(\phi,\vec{u}) \in U$ of the ultra-weak first order formulation
$$
B_\kappa \mathbbm{u}=q.
$$
Unfortunately, the dual norm  $\|\cdot\|_{V_{\mp,\kappa}'}$ cannot be exactly evaluated, so that even for a finite dimensional $U^\delta$ this ideal $\mathbbm{u}^\delta$ is not computable. As explained in the next subsection, we will content ourselves with a computable approximation from $U^\delta$ that is \emph{quasi-best}.

\subsection{`Practical' method}
The minimizer $\mathbbm{u}^\delta$ from \eqref{ls} is solution of the Euler-Lagrange equations
\be \label{EL}
\langle q-B_\kappa \mathbbm{u}^\delta,B_\kappa \undertilde{\mathbbm{u}}^\delta \rangle_{V_{\mp,\kappa}'}=0 \quad (\undertilde{\mathbbm{u}}^\delta \in U^\delta).
\ee
With the Riesz isometry $R_\kappa \in \Lis(V_{\mp,\kappa}',V_{\mp,\kappa})$, let $\mathbbm{v}:=R_\kappa(q-B_\kappa \mathbbm{u}^\delta)$, i.e.,
$$
\langle \mathbbm{v}, \undertilde{\mathbbm{v}}\rangle_{V_{\mp,\kappa}}=(q-B_\kappa \mathbbm{u}^\delta)(\undertilde{\mathbbm{v}}) \quad (\undertilde{\mathbbm{v}} \in V_{\mp,\kappa}).
$$
Then an equivalent formulation of \eqref{EL} reads as
$$
0=\langle R_\kappa(q-B_\kappa \mathbbm{u}^\delta),R_\kappa B_\kappa \undertilde{\mathbbm{u}}^\delta \rangle_{V_{\mp,\kappa}}=
\langle  \mathbbm{v}, R_\kappa B_\kappa \undertilde{\mathbbm{u}}^\delta \rangle_{V_{\mp,\kappa}}=\overline{(B_\kappa \undertilde{\mathbbm{u}}^\delta)(\mathbbm{v})} \quad (\undertilde{\mathbbm{u}}^\delta \in U^\delta).
$$
Upon using \eqref{eq:optimal test norm} and $(B_\kappa \mathbbm{w})(\mathbbm{z})=\langle B_\kappa' \mathbbm{z},\mathbbm{w}\rangle_U$, we conclude that  the solution $\mathbbm{u}^\delta=(\phi^\delta,\vec{u}^\delta)$ of \eqref{ls} is the second component of $(\mathbbm{v},\mathbbm{u}^\delta) \in V_{\mp} \times U^\delta$ that solves the system
$$
\left\{\hspace*{-0.5em}
\begin{array}{lcll}
\langle B_\kappa'\mathbbm{v}, B_\kappa'\undertilde{\mathbbm{v}}\rangle_U +
\langle \mathbbm{u}^\delta, B_\kappa'\undertilde{\mathbbm{v}}\rangle_U 
& \!\!= \!\!& q(\undertilde{\mathbbm{v}}) & (\undertilde{\mathbbm{v}} \in V_{\mp}),\\
\langle B_\kappa'\mathbbm{v},\undertilde{\mathbbm{u}}^\delta\rangle_U
& \!\!=\!\! & 0 & (\undertilde{\mathbbm{u}}^\delta \in U^\delta).
\end{array}
\right.\hspace*{-1.2em}
$$%

Now consider a \new{closed subspace} $V_{\mp}^\delta \subset V_{\mp}$ for which
$$
\gamma^\delta_\kappa:=\inf_{0 \neq \undertilde{\mathbbm{u}}^\delta \in U^\delta} \sup_{0 \neq \undertilde{\mathbbm{v}}^\delta \in V_{\mp}^\delta} \frac{|
(B_\kappa \undertilde{\mathbbm{u}}^\delta)(\undertilde{\mathbbm{v}}^\delta)|}{\|\undertilde{\mathbbm{u}}^\delta\|_U\|\undertilde{\mathbbm{v}}^\delta\|_{V_{\new{\mp},\kappa}}}>0,
$$
Then finding $(\mathbbm{v}^\delta,\hat{\mathbbm{u}}^{\delta}) \in V_{\mp}^\delta \times U^\delta$ such that 
\be \label{saddle-discrete}
\left\{\hspace*{-0.5em}
\begin{array}{lcll}
\langle B_\kappa'\mathbbm{v}^\delta, B_\kappa'\undertilde{\mathbbm{v}}^\delta\rangle_U +
\langle \hat{\mathbbm{u}}^{\delta}, B_\kappa'\undertilde{\mathbbm{v}}^\delta\rangle_U 
& \!\!= \!\!& q(\undertilde{\mathbbm{v}}^\delta) & (\undertilde{\mathbbm{v}}^\delta \in V_{\mp}^\delta),\\
\langle B_\kappa'\mathbbm{v}^\delta,\undertilde{\mathbbm{u}}^\delta\rangle_U
& \!\!=\!\! & 0 & (\undertilde{\mathbbm{u}}^\delta \in U^\delta).
\end{array}
\right.\hspace*{-1.2em}
\ee
has a unique solution, \new{that, for finite dimensional $U^\delta$ and $V^\delta_{\new{\mp}}$}, is computable. For notational convenience from here on we will denote $\hat{\mathbbm{u}}^{\delta}$ again by $\mathbbm{u}^\delta=(\phi^\delta,\vec{u}^\delta)$.

The following result shows how much the error in this $\mathbbm{u}^\delta$ \emph{can}, and for some $\mathbbm{u}$ \emph{will be} larger than the error in the best approximation from $U^\delta$ to $\mathbbm{u}$.

\begin{theorem} \label{thm1} It holds that
$$
\sup_{\mathbbm{u} \in U\setminus U^\delta}
\frac{
\|\mathbbm{u}-\mathbbm{u}^\delta\|_U
}{
\inf_{\undertilde{\mathbbm{u}}^\delta \in U^\delta} \|\mathbbm{u}-\undertilde{\mathbbm{u}}^\delta\|_U
}= \frac{1}{\gamma^\delta_\kappa}.
$$
\end{theorem}

\begin{proof} \new{With $B_\kappa^\delta\colon U^\delta \rightarrow {V_{\mp}^\delta}'$ defined by
$(B_\kappa^\delta \undertilde{\mathbbm{u}}^\delta)(\undertilde{\mathbbm{v}}^\delta)=\langle \undertilde{\mathbbm{u}}^\delta, B_\kappa'\undertilde{\mathbbm{v}}^\delta \rangle_U$, and
$G_\kappa^\delta \colon V_{\mp}^\delta \rightarrow {V_{\mp}^\delta}'$ defined by $(G_\kappa^\delta  \mathbbm{v}^\delta)(\undertilde{\mathbbm{v}}^\delta)=\langle B_\kappa'\mathbbm{v}^\delta, B_\kappa'\undertilde{\mathbbm{v}}^\delta\rangle_U$}, \eqref{saddle-discrete} reads as the system 
$\new{G_\kappa^\delta} \mathbbm{v}^\delta+B_\kappa^\delta \mathbbm{u}^\delta=q|_{V_{\mp}^\delta}$, ${B_\kappa^\delta}' \mathbbm{v}^\delta=0$.
So $\mathbbm{u}^\delta$ is the unique solution of the Schur complement ${B_\kappa^\delta}' (\new{G_\kappa^\delta})^{-1} (B_\kappa^\delta \mathbbm{u}^\delta-q|_{V_{\mp}^\delta})=0$, or equivalently, of the Petrov-Galerkin discretisation
$$
(B_\kappa \new{\mathbbm{u}^\delta})(\undertilde{\mathbbm{v}}^\delta )=q(\undertilde{\mathbbm{v}}^\delta ) \quad (\undertilde{\mathbbm{v}}^\delta \in \ran (\new{G_\kappa^\delta})^{-1}B_\kappa^\delta).
$$

Let $\undertilde{\mathbbm{z}}^\delta \in (\ran (\new{G_\kappa^\delta})^{-1}B_\kappa^\delta)^{\perp_{\langle \cdot,\cdot \rangle_{V_{\new{\mp},\kappa}}}} \subset V_{\mp}^\delta$ and $\undertilde{\mathbbm{u}}^\delta \in U^\delta$. 
With $\mathbbm{y}^\delta:=(\new{G_\kappa^\delta})^{-1}B_\kappa^\delta\undertilde{\mathbbm{u}}^\delta$, we find
$$
(B_\kappa \undertilde{\mathbbm{u}}^\delta)( \undertilde{\mathbbm{z}}^\delta)= (\new{G_\kappa^\delta} \mathbbm{y}^\delta)( \undertilde{\mathbbm{z}}^\delta)=
\langle {B_\kappa^\delta}' \mathbbm{y}^\delta,{B_\kappa^\delta}' \undertilde{\mathbbm{z}}^\delta\rangle_U=
\langle \mathbbm{y}^\delta, \undertilde{\mathbbm{z}}^\delta\rangle_{V_{\mp},\kappa}=0.
$$
We infer that consequently
 the value of inf-sup constant $\gamma^\delta_\kappa$ remains unchanged when in its definition the space $V_{\mp}^\delta$ is replaced by its subspace $\ran (\new{G_\kappa^\delta})^{-1}B_\kappa^\delta$. 
Since furthermore $B_\kappa \in \Lis(U,V_{\mp}')$ is an isometry, an application of \cite[Remark 3.2]{249.99} concerning the error in Petrov-Galerkin discretizations 
completes the proof.
\end{proof}

\begin{remark} \label{remmie} One easily infers that $\mathbbm{u} \mapsto  \mathbbm{u}^\delta$ is a projector onto $U^\delta$ with $\|\mathbbm{u} \mapsto  \mathbbm{u}^\delta\|_{\cL(U,U)} \leq \frac{1}{\gamma^\delta_\kappa}$. Even $\|\mathbbm{u} \mapsto  \mathbbm{u}^\delta\|_{\cL(U,U)} = \frac{1}{\gamma^\delta_\kappa}$ is valid (see \cite[Thm.~3.1]{249.99}).
\end{remark}

\begin{remark}
In view of the equivalent definition of $\mathbbm{u}^\delta$ as the solution of a Petrov-Galerkin discretisation, we refer to $U^\delta$ as the \emph{trial space}, and $V_{\mp}^\delta$ as the \emph{test space} (actually its subspace $\new{\ran (G_\kappa^\delta)^{-1}B_\kappa^\delta}$ is the test space in the Petrov-Galerkin discretisation, and therefore in the literature $V_{\mp}^\delta$ is  sometimes called the test search space).
\end{remark}

\begin{remark}  Besides being (a component of) the solution of a saddle point system, and being a Petrov-Galerkin solution, a third equivalent definition of $\mathbbm{u}^\delta$ is that by
$$
\mathbbm{u}^\delta:=\argmin_{\mathbbm{w}^\delta \in U^\delta} \sup_{0 \neq \mathbbm{z}^\delta \in V^\delta_{\mp}} \frac{|(q- B_\kappa \mathbbm{w}^\delta)(\mathbbm{z}^\delta)|}{\|\mathbbm{z}^\delta\|_{V_{\mp,\kappa}}},
$$ 
so that it is appropriate to call $\mathbbm{u}^\delta$ the solution of a minimal residual or least squares discretisation.
\end{remark}

We will refer to the `practical' method discussed in this subsection as the \emph{ultra-weak FOSLS}, or shortly, \emph{FOSLS} method or discretisation, noting that its solution not only depends on $U^\delta$ but also on $V^\delta_{\mp}$. \medskip

In view of Theorem~\ref{thm1} our goal is to find $V^\delta_{\mp}=V^\delta_{\mp}(U^\delta)$, preferably with $\dim V^\delta_{\mp} \lesssim \dim U^\delta$,
such that inf-sup stability $\gamma_\kappa^\delta \gtrsim 1$ is valid \emph{uniformly} in $U^\delta$ and $\kappa$, so that $\mathbbm{u}^\delta$ is a quasi-best approximation from $U^\delta$.

To show uniform inf-sup stability, one can use the characterization of $\gamma_\kappa^\delta$ as being the reciprocal of the $\cL\big((V_{\mp},\|\cdot\|_{V_{\mp},\kappa}), (V_{\mp},\|\cdot\|_{V_{\mp},\kappa})\big)$-norm of the Fortin interpolator\footnote{A linear mapping $\Pi^\delta\colon V_{\mp} \rightarrow V_{\mp}^\delta$ with $(B_\kappa U^\delta)(\ran (\identity-\Pi^\delta))=0$.} with smallest norm.
\new{For our main result presented in Sect.~\ref{sec:practical},} we will rely on an alternative expression for $\gamma_\kappa^\delta$ in terms of approximability from $V_{\mp}^\delta$ of the solution of an adjoint problem with a forcing function from $U^\delta$ which we recall next.

\begin{proposition}[{\cite[Proposition 2.5]{35.8565}}] \label{proppie} It holds that
\be \label{71}
\new{\sup_{0 \neq \mathbbm{w}^\delta \in U^\delta}
\frac{
\inf_{\mathbbm{v}^\delta \in V_\mp^\delta} \|B_\kappa'\big[{B_\kappa'}^{-1} \mathbbm{w}^\delta -\mathbbm{v}^\delta\big]\|_U^2
}
{
\|\mathbbm{w}^\delta\|_U^2
}
=1-(\gamma_k^\delta)^2.}
\ee
\end{proposition}

\begin{proof} \new{Let $P^\delta\colon V_\mp \rightarrow V_\mp^\delta$ be the $\langle\cdot,\cdot\rangle_{V_\mp,\kappa}$-orthogonal projector onto $V_{\mp}^\delta$.}
From
\begin{align*}
&\inf_{0 \neq \mathbbm{w}^\delta \in U^\delta} \frac{\|P^\delta {B_\kappa'}^{-1} \mathbbm{w}^\delta\|_{V_{\mp,\kappa}}}{\| \mathbbm{w}^\delta\|_U}
=
\inf_{0 \neq \mathbbm{w}^\delta \in U^\delta} \sup_{0 \neq \mathbbm{v}^\delta \in V_{\mp}^\delta} \frac{|\langle P^\delta {B_\kappa'}^{-1} \mathbbm{w}^\delta,\mathbbm{v}^\delta\rangle_{V_{\mp,\kappa}}|}{\| \mathbbm{w}^\delta\|_U\|\mathbbm{v}^\delta\|_{V_{\mp,\kappa}}}=\\
&\inf_{0 \neq \mathbbm{w}^\delta \in U^\delta} \sup_{0 \neq\mathbbm{v}^\delta \in V_{\mp}^\delta} \frac{|\langle  {B_\kappa'}^{-1} \mathbbm{w}^\delta,\mathbbm{v}^\delta\rangle_{V_{\mp,\kappa}}|}{\| \mathbbm{w}^\delta\|_U\|\mathbbm{v}^\delta\|_{V_{\mp,\kappa}}}=
\inf_{0 \neq \mathbbm{w}^\delta \in U^\delta} \sup_{0 \neq\mathbbm{v}^\delta \in V_{\mp}^\delta}
\frac{|\langle \mathbbm{w}^\delta,B_\kappa' \mathbbm{v}^\delta\rangle_U|}{\| \mathbbm{w}^\delta\|_U \|\mathbbm{v}^\delta\|_{V_{\mp,\kappa}}}=\gamma_\kappa^\delta,
\end{align*}
we obtain that
\be \label{70}
\begin{split}
&\sup_{0 \neq \mathbbm{w}^\delta \in U^\delta}\frac{\|(\identity -P^\delta){B_\kappa'}^{-1} \mathbbm{w}^\delta\|^2_{V_{\mp,\kappa}}}{\|\mathbbm{w}^\delta\|_U^2}=
\sup_{0 \neq \mathbbm{w}^\delta \in U^\delta}
\frac{\|{B_\kappa'}^{-1} \mathbbm{w}^\delta\|^2_{V_{\mp,\kappa}}-\|P^\delta{B_\kappa'}^{-1} \mathbbm{w}^\delta\|^2_{V_{\mp,\kappa}}
}{\|\mathbbm{w}^\delta\|_U^2}\\
&=
1-\inf_{0 \neq \mathbbm{w}^\delta \in U^\delta}
\frac{\|P^\delta{B_\kappa'}^{-1} \mathbbm{w}^\delta\|^2_{V_{\mp,\kappa}}
}{\|\mathbbm{w}^\delta\|_U^2} = 1-(\gamma^\delta_\kappa)^2.
\end{split}
\ee
\new{Noticing that the first expression in \eqref{70} equals the left-hand side in \eqref{71}, the proof is completed.}
\end{proof}

\section{Verification of uniform inf-sup stability} \label{sec:infsupstab}

\subsection{The case of $\kappa h_{\delta,\min}$  sufficiently large}
As we will see, for $U^\delta$ being a finite element space of fixed polynomial degree, it is fairly easy to select $V^\delta_{\mp}$ with $\dim V^\delta_{\mp} \lesssim \dim U^\delta$
that gives inf-sup stability, and so quasi-optimality of our FOSLS discretisation, assuming that the product of $\kappa$ and the minimal mesh-size $h_{\delta,\min}$ is sufficiently large. In this regime standard Galerkin discretisations are known not to be quasi-optimal.

\new{Notice, however, that physically relevant solutions of the Helmholtz problem generally exhibit globally oscillations with wavelength $\sim \frac{2\pi}{\kappa}$, so that from an approximation point of view, the case of $\kappa h_{\delta,\min}$ being sufficiently large is not very relevant.
The reasons to nevertheless study this case are two-fold.
First, mathematically it is possible that the Helmholtz solution is smooth, so that meaningful approximations do exist on coarse meshes.
Second, restricting ourselves to quasi-uniform meshes, in Sect.~\ref{sec:practical} we show quasi-optimality in the most relevant case of $\kappa h_{\delta}$ being uniformly bounded.
Together with the result of the current subsection, we conclude quasi-optimality of the FOSLS solution 
on \emph{any} quasi-uniform mesh and $\kappa \geq \kappa_0$.}

\begin{theorem} \label{thm:coarse} For a conforming subdivision $\tria^\delta$ of $\Omega$ into $d$-simplices, let
$$
h_{\delta,\min}:=\min_{K \in \tria^\delta} |K|^{1/d}.
$$
For $p \in \N_0$, with $\cS_p^{-1}(\tria^\delta):=\{v \in L_2(\Omega)\colon v|_K \in P_p(K)\,\,(K \in \tria^{\new{\delta}})\}$, and
$\cS_p^{0}(\tria^\delta):=\cS^{-1}(\tria^{\new{\delta}}) \cap C(\overline{\Omega})$, let
$$
U^\delta:=\cS_p^{-1}(\tria^\delta)^{d+1},\quad
V_{\mp}^\delta:=(\cS^0_{p+d+1}(\tria^\delta) \times \RT_{p+1}(\tria^\delta)) \cap V_{\mp}.
$$
Then for $\kappa h_{\delta,\min}$ sufficiently large, it holds that
$\gamma_\kappa^\delta \gtrsim 1$, only dependent on $d$, the shape regularity of $\tria^\delta$, and here also on $p$.\footnote{\new{Both the lower bound on $\kappa h_{\delta,\min}$ and the constant hidden in the inequality $\gamma_\kappa^\delta \gtrsim 1$ may depend on $p$. For this less relevant unresolved regime, we have not made an attempt to construct $V_{\mp}^\delta$ that gives results that hold uniformly in $p$.}}
\end{theorem}

\begin{proof} For a $d$-simplex $K$, and with $b_K\colon K \rightarrow \R_+$ being the product of its $d+1$ barycentric coordinates, it holds that $P_{p+d+1}(K) \cap H^1_0(K)=b_K P_p(K)$.
For $\new{\phi} \in P_p(K)$, it holds that $\int_K b_K |\new{\phi}|^2\,dx \gtrsim \int_K |\new{\phi}|^2\,dx \gtrsim \int_K b^2_K |\new{\phi}|^2\,dx$, which shows that
\be \label{81}
\inf_{0\neq \phi \in  P_p(K)} \sup_{0 \neq \eta \in P_{p+d+1}(K) \cap H^1_0(K)} \frac{\Re \langle \phi,\eta\rangle_{L_2(K)}}{\|\phi\|_{L_2(K)}\|\eta\|_{L_2(K)}}\gtrsim 1,
\ee
independent of $K$.

Let $\hat{K}$ some fixed reference $d$-simplex. Since an $\new{\vec{v}} \in \RT_{p+1}(\hat{K}) \cap H_0(\divv;\hat{K})$ is uniquely determined by the DoFs 
$\new{\vec{u}} \mapsto \int_{\hat{K}} \new{\vec{v}} \cdot \new{\vec{u}}\,dx$ ($\new{\vec{u}} \in P_{\new{p}}(\hat{K})^{\new{d}}$), it holds that
\be \label{82}
 \inf_{0 \neq \new{\vec{u}} \in P_{p}(\hat{K})^d} \sup_{0\neq \new{\vec{v}} \in \RT_{p+1}(\hat{K}) \cap H_0(\divv;\hat{K})} \frac{\Re \langle \new{\vec{u}},\new{\vec{v}}\rangle_{L_2(\hat{K})^d}}{\|\new{\vec{u}}\|_{L_2(\hat{K})^d}\|\new{\vec{v}}\|_{L_2(\hat{K})^d}}\gtrsim 1.
 \ee
By employing a contravariant Piola transformation \new{(e.g.~\cite[Definition~9.8]{70.97}) in order to preserve the vanishing normal components of the vector field $\new{\vec{v}}$}, this results carries over to each uniformly shape regular $d$-simplex $K$.

\new{By selecting 
$\|\eta\|_{L_2(K)}=\|\phi\|_{L_2(K)}$
and
$\|\new{\vec{v}}\|_{L_2(K)^d}=\|\new{\vec{u}}\|_{L_2(K)^d}$
 in \eqref{81} and \eqref{82}, and summing over $K \in \tria^\delta$}, using $|\cdot|\geq \Re\,\cdot$  we conclude that with
 \begin{align*}
\hat{V}^\delta:=\big\{&\mathbbm{v} \in L_2(\Omega)^{d+1}\colon\\
 & \mathbbm{v}|_K \in (P_{p+d+1}(K) \cap H^1_0(K)) \times (\RT_{p+1}(K) \cap H_0(\divv;K))\,\new{(K \in \tria^\delta)}\big\} \subset V^\delta_{\mp},
\end{align*}
it holds that
\be \label{r1}
 \inf_{0 \neq (\phi,\vec{u}) \in U^\delta} \sup_{0\neq (\eta,\vec{v}) \in \hat{V}^\delta}\frac{\big|\langle (\phi,\vec{u}),(\eta,\vec{v})\rangle_{U}\big|}{\|(\phi,\vec{u})\|_{U}\|(\eta,\vec{v})\|_{U}}\gtrsim 1.
\ee

Using an inverse inequality on $\hat{V}^\delta$, for $(\eta,\vec{v}) \in \hat{V}^\delta$ we have
\be \label{r2}
\|B_\kappa'(\eta,\vec{v})+(\eta,\vec{v})\|_{U}=\tfrac{1}{\kappa} \|(-\divv \vec{v},\nabla \eta)\|_{U} \lesssim \tfrac{1}{\kappa h_{\delta,\min}} \|(\eta,\vec{v})\|_{U}.
\ee

From \eqref{r1} and \eqref{r2}, by estimating
\begin{align*}
\big|\langle (\phi,\vec{u}), B_\kappa'(\eta,\vec{v})\rangle_{U}\big|& =
\big|-\langle (\phi,\vec{u}), (\eta,\vec{v}) \rangle_{U}+
\langle (\phi,\vec{u}), B_\kappa'(\eta,\vec{v})+ (\eta,\vec{v})\rangle_{U}\big|\\
& \geq
\big|\langle (\phi,\vec{u}), (\eta,\vec{v}) \rangle_{U}\big| -\|(\phi,\vec{u})\|_{U}\|B_\kappa'(\eta,\vec{v}) +(\eta,\vec{v})\|_{U}
\end{align*}
one infers inf-sup stability \new{$\gamma_\kappa^\delta \gtrsim 1$} for the the pair $(U^\delta,\hat{V}^\delta)$ when $\kappa h_{\delta,\min}$ is sufficiently large. Obviously this inf-sup stability remains valid when we replace $\hat{V}^\delta$ by 
$V^\delta_{\mp}$.
\end{proof}

\begin{remark} The results in this and the next subsection about inf-sup stability are shown for pairs $(U^\delta,V^\delta_\mp)$ where $U^\delta$ is a space of discontinuous piecewise polynomials. Obviously, these results carry over to the case when such a space of discontinuous piecewise polynomials is replaced by the smaller space 
 of corresponding continuous piecewise polynomials.
\end{remark}

\begin{remark} \new{The stability demonstrated in Theorem~\ref{thm:coarse} clearly also holds true when $V_{\mp}^\delta$ from that theorem is replaced by the larger space $V_{\mp}^\delta:=(\cS^0_{\tilde{p}}(\tria^\delta) \times \RT_{\tilde{p}}(\tria^\delta)) \cap V_{\mp}$, where $\tilde{p}=p+d+1$. Test spaces of this type will also be considered in the next subsection.
By applying standard bases for $\cS^0_{\tilde{p}}(\tria^\delta)$ and $\RT_{\tilde{p}}(\tria^\delta)$, a basis for $V_{\mp}^\delta$ is obtained by eliminating the common DoFs of $\RT_{\tilde{p}}(\tria^\delta)$ associated to element faces on $\Gamma_N \cup \Gamma_R$ by imposing for $(\eta,\vec{v}) \in \cS^0_{\tilde{p}}(\tria^\delta) \times \RT_{\tilde{p}}(\tria^\delta)$, $\vec{v}\cdot\vec{n}=0$ on $\Gamma_N$ and $\vec{v}\cdot\vec{n}=\pm i \eta$ on $\Gamma_R$.}
\end{remark}

\subsection{The case of $\kappa h_\delta$ uniformly bounded} \label{sec:practical}
In this subsection, we assume that the conforming, uniformly shape regular subdivisions $\tria^\delta$ of $\Omega$ into $d$-simplices are \emph{quasi-uniform}, i.e., that
$$
h_\delta:=h_{\delta,\min} \eqsim h_{\delta,\max}:=\max_{K \in \tria^\delta} |K|^{1/d}.
$$

\new{In view of Proposition~\ref{proppie}, the next lemma will lead to a sufficient condition for the uniform inf-sup stability $\gamma_\kappa^\delta \gtrsim 1$.

\begin{lemma} \label{lemma:extra} For some $p \in \N_0$, let $U^\delta \subseteq \cS_p^{-1}(\tria^\delta)^{d+1}$. 
Let $V^\delta_{\mp}$ be some closed subspace of $V_{\mp}$.  Then for given $s \in (0,\frac12)$, it holds that
\begin{align*}
\sup_{0 \neq \mathbbm{w}^\delta \in U^\delta} &
\frac{
\inf_{\mathbbm{v}^\delta \in V_\mp^\delta} \|B_\kappa'\big[{B_\kappa'}^{-1} \mathbbm{w}^\delta -\mathbbm{v}^\delta\big]\|_U
}
{
\|\mathbbm{w}^\delta\|_U
}
\\
&\hspace{3em}\lesssim \Big((p+1)^2
\sup_{0 \neq \mathbbm{w}\in H^1_0(\Omega)^{d+1}}
\frac{
\inf_{\mathbbm{v}^\delta \in V_\mp^\delta} \|B_\kappa'\big[{B_\kappa'}^{-1} \mathbbm{w} -\mathbbm{v}^\delta\big]\|_U
}
{
\sqrt{\|\mathbbm{w}\|^2_U+h_\delta\|\mathbbm{w}\|_{H^1(\Omega)^{d+1}}^2}
}
\Big)^s.
\end{align*}
\end{lemma}

\begin{proof}
Clearly, it holds that
$$
\sup_{0 \neq \mathbbm{w}\in U}
\frac{
\inf_{\mathbbm{v}^\delta \in V_\mp^\delta} \|B_\kappa'\big[{B_\kappa'}^{-1} \mathbbm{w} -\mathbbm{v}^\delta\big]\|_U
}
{
\|\mathbbm{w}\|_U
} \leq 1.
$$
We set $U_1:=H_0^1(\Omega)^{d+1}$ equipped with norm $
\|\cdot\|^2_{U_1}:=
\sqrt{\|\cdot\|^2_{U}+h_\delta^2\|\cdot\|^2_{H^1(\Omega)^{d+1}}}.
$
Recalling that $\inf_{\mathbbm{v}^\delta \in V_\mp^\delta} \|B_\kappa'\big[{B_\kappa'}^{-1} \cdot -\mathbbm{v}^\delta\big]\|_U=
\|(\identity-P^\delta){B_\kappa'}^{-1}\cdot\|_{V_{\mp,\kappa}}$,  from the Riesz-Thorin interpolation theorem it follows that with $[U,U_1]_{s,2}$ being the interpolation space (e.g.~\cite{20}),
\be \label{90}
\sup_{0 \neq\mathbbm{w} \in [U,U_1]_{s,2}} \hspace*{-1.5em}
\frac{\inf_{\mathbbm{v}^\delta \in V_\mp^\delta} \|B_\kappa'\big[{B_\kappa'}^{-1} \mathbbm{w} -\mathbbm{v}^\delta\big]\|_U}
{\|\mathbbm{w}\|_{[U,U_1]_{s,2}}} \leq \Big(
\sup_{0 \neq \mathbbm{w}\in U_1} \hspace*{-1em}
\frac{\inf_{\mathbbm{v}^\delta \in V_\mp^\delta} \|B_\kappa'\big[{B_\kappa'}^{-1} \mathbbm{w}-\mathbbm{v}^\delta\big]\|_U}{\|\mathbbm{w}\|_{U_1}}
\Big)^s.
\ee

Noticing that $U_1=U \cap \big\{H_0^1(\Omega)^{d+1}, h_\delta \|\cdot\|_{H^1(\Omega)^{d+1}}\big\}$, it follows from \cite[Ch.~1, Thm.~13.1]{185} 
that
$$
[U,U_1]_{s,2} \simeq U \cap \Big[U,\big\{H_0^1(\Omega)^{d+1}, h_\delta \|\cdot\|_{H^1(\Omega)^{d+1}}\big\}\Big]_{s,2}.
$$
Because $s \in (0,\frac12)$, $[L_2(\Omega),H_0^1(\Omega)]_{s,2} \simeq H^s(\Omega)$ (e.g.~\cite[Ch.~1, Thm.~11.1\&11.6]{185}), 
and thus $[L_2(\Omega),\{H_0^1(\Omega),h_\delta\|\cdot\|_{H^1(\Omega)}\}]_{s,2} \simeq \{H^s(\Omega),h^s_\delta\|\cdot\|_{H^s(\Omega)}\}$.
We conclude that $[U,U_1]_{s,2}$ equals $H^s(\Omega)^{d+1}$ equipped with norm
\be \label{91}
\|\cdot\|_{[U,U_1]_{s,2}} \eqsim \sqrt{\|\cdot\|_U^2+h^{2s}_\delta\|\cdot\|_{H^s(\Omega)^{d+1}}^2}\,.
\ee

Again because $s \in (0,\frac12)$, it holds that $\cS_p^{-1}(\tria^\delta) \subset H^s(\Omega)$, and as shown in \cite[Thm.~3.5]{168.835}, the following inverse inequality is valid
\be \label{92}
\|w^\delta\|_{H^s(\Omega)} \lesssim (h_\delta^{-1} (p+1)^2)^{s}  \|w^\delta\|_{L_2(\Omega)} \quad (w^\delta \in \cS_p^{-1}(\tria^\delta)).
\ee

From \eqref{90}, \eqref{91}, and \eqref{92}, one easily completes the proof.
\end{proof}}


\new{To continue,} we equip  $H^1(\Omega)$ and $H(\divv;\Omega)$ with the $\kappa$-dependent norms
\begin{align*}
\|\cdot\|_{1,\kappa}& :=\sqrt{\|\cdot\|_{L_2(\Omega)}^2+\tfrac{1}{\kappa^{2}} |\cdot|_{H^1(\Omega)}^2}\,\,,\\
\|\cdot\|_{\divv,\kappa}& :=\sqrt{\|\cdot\|_{L_2(\Omega)^{\new{d}}}^2+\tfrac{1}{\kappa^{2}} \|\divv \cdot\|_{L_2(\Omega)}^2}\,\,,
\intertext{and $H^{\frac12}(\partial\Omega)$ and $H^{-\frac12}(\partial\Omega)$ with the $\kappa$-dependent quotient norms}
\|g\|_{\frac12,\kappa}&:=\inf\{\|v\|_{1,\kappa}\colon v \in H^1(\Omega),\,v|_{\partial\Omega}=g\},\\
\|g\|_{-\frac12,\kappa}&:=\inf\{\|\vec{v}\|_{\divv,\kappa}\colon \vec{v} \in H(\divv;\Omega),\,\vec{v}\cdot\vec{n}|_{\partial\Omega} =g\}.
\end{align*}
An application of the triangle inequality gives that for $(\eta,\vec{v}) \in V_{\mp}$,
\be \label{triangle}
\|(\eta,\vec{v})\|^2_{V_{\mp},\kappa} \leq 2( \|\eta\|^2_{1,\kappa}+\|\vec{v}\|^2_{\divv,\kappa}).
\ee
\begin{remark} \label{crude}
The norm $\|\cdot\|_{V_{\mp},\kappa}$ on $V_{\mp}$ extends to a semi-norm on $H^1(\Omega) \times H(\divv;\Omega)$.
For $\eta$ with $-\Delta \eta-\kappa^2 \eta=0$, e.g.~a plane wave $\eta(\vec{x})=e^{-i\kappa \vec{r}\cdot\vec{x}}$ for some $\|\vec{r}\|=1$, and $\vec{v}=\kappa^{-1} \nabla \eta$, it holds that $\|(\eta,\vec{v})\|_{V_{\mp},\kappa}=0$.
The boundary condition included in the definition of $V_\mp$ distinguishes this space from $H^1(\Omega) \times H(\divv;\Omega)$.
Recalling that $\|\cdot\|_{V_{\mp},\kappa}=\|B_\kappa\cdot\|_U$, and ``$B_\infty$''$=\identity$, the restriction imposed by these boundary conditions vanish in the limit. Consequently, it can be expected that the inequality \eqref{triangle} can be increasingly crude when $\kappa \rightarrow \infty$.
\end{remark}

The following lemma is trivial for fixed $\kappa$, but it requires a proof to demonstrate its uniform validity for $\kappa \geq \kappa_0>0$.

\begin{lemma} \label{lem3} Let $\Omega \subset \R^d$ be a bounded Lipschitz polytope in which $d$ edges meet at each vertex.
Then
\be \label{bdrnorms}
\|\cdot\|_{-\frac12,\kappa} \lesssim \|\cdot\|_{\frac12,\kappa} \quad \text{on } H^{\frac12}(\partial\Omega) \,\,(\text{uniformly in } \kappa \geq \kappa_0>0).
\ee
\end{lemma}

\begin{proof} 
It suffices to show that for each $\phi \in H^1(\Omega)$ there exists a $\vec{v} \in H(\divv;\Omega)$ with $\vec{v}\cdot\vec{n}|_{\partial\Omega}=\phi|_{\partial\Omega}$,
and $\|\vec{v}\|_{\divv,\kappa} \lesssim \|\phi\|_{1,\kappa}$.

With $J$ being the number of vertices, let $\{\chi_i \colon 1 \leq i \leq J\}$ be a partition of unity of an open neighborhood $\hat{\Omega}$ of $\Omega$
 where 
$\chi_i \in W^1_\infty(\hat{\Omega})$, $\supp \chi_i$ is connected, and $\chi_i$ vanishes on all facets of $\Omega$ that do not emanate from the $i$th vertex.

Let $\vec{n}_{i,j}$, $1 \leq j \leq d$, be the normals to the faces that intersect at the $i$th vertex.
With $\vec{r}_i$ such that $\vec{r}_i \cdot \vec{n}_{i,j}=1$ ($1 \leq j \leq d$), let $\vec{v}_i:=\chi_i \phi \vec{r}_i$.
Then $\vec{v}_i \cdot \vec{n}|_{\partial\Omega}=\chi_i \phi|_{\partial\Omega}$, and so $\vec{v}:=\sum_i \vec{v}_i$ satisfies $\vec{v}\cdot\vec{n}|_{\partial\Omega}=\phi|_{\partial\Omega}$.

Furthermore,
$\|\vec{v}\|_{L_2(\Omega)^d} \leq \sum_i \|\vec{v}_i\|_{L_2(\Omega)^d} \lesssim \sum_i  \|\chi_i \phi\|_{L_2(\Omega)} \lesssim \|\phi\|_{L_2(\Omega)}$
and
$$
\|\divv \vec{v}\|_{L_2(\Omega)} \leq \sum_i \|\divv \vec{v}_i\|_{L_2(\Omega)} \lesssim \sum_i  \sum_k |\vec{r}_{\new{i}}\cdot\vec{e}_k|\|\partial_k (\chi_i \phi)\|_{L_2(\Omega)}
\lesssim \|\phi\|_{L_2(\Omega)}+|\phi|_{H^1(\Omega)},
$$
and so $\|\vec{v}\|_{\divv,\kappa} \lesssim \|\phi\|_{1,\kappa}$.
\end{proof}

\begin{remark} The validity of \eqref{bdrnorms} extends to other situations. For example, it can be shown when $\partial\Omega$ is $C^2$.
\end{remark}

For $\tilde{\tria}^\delta$ a conforming subdivision of $\Omega$ into $d$-simplices, we take
\be \label{301}
V_{\mp}^\delta:=(\cS_{\tilde{p}}^0(\tilde{\tria}^\delta) \times \RT_{\tilde{p}}(\tilde{\tria}^\delta)) \cap V_{\mp},
\ee
under the condition that $\tilde{p} \in \N$ and $\tilde{\tria}^\delta$ are selected such that
\be \label{MS}
\sup_{0 \neq f \in L_2(\Omega)} \inf_{\eta^\delta \in \cS_{\tilde{p}}^0(\tilde{\tria}^\delta) } \frac{{\| (L'_\kappa})^{-1} f -\eta^\delta\|_{1,\kappa}}{\|f\|_{L_2(\Omega)}} \lesssim \frac{h_\delta \kappa}{\tilde{p}},
\ee
Recall that $h_\delta$ is the mesh-size of $\tria^\delta$, and that the adjoint $L_\kappa'$ equals $L_\kappa$ from \eqref{2nd} with $\pm$ replaced by $\mp$.
\medskip

Before continuing, we discuss condition \eqref{MS} in some detail.
In the seminal work \cite{202.8}, the left-hand side of \eqref{MS} is denoted as $\kappa \eta(\cS_{\tilde{p}}^0(\tilde{\tria}^\delta))$. It is shown that $\kappa \eta(\cS_{\tilde{p}}^0(\tilde{\tria}^\delta))$ being less than some sufficiently small constant implies quasi-optimality w.r.t.~$\|\cdot\|_{1,\kappa}$ of the Galerkin discretisation on $\cS_{\tilde{p}}^0(\tilde{\tria}^\delta)$ of the usual second order variational formulation \eqref{2nd} of the Helmholtz problem.

For the case that $d \in \{2,3\}$, $\partial\Omega=\Gamma_R$, and $\Omega$ being a bounded domain with an analytic boundary, in \cite{202.8} it has been shown that \eqref{MS} holds for $\tilde{\tria}^\delta=\tria^\delta$ when $\kappa h_{\delta} \lesssim 1$ and $\tilde{p}^{-1} \log \kappa$ is sufficiently small (one verifies that under these conditions, the right-hand side of \cite[(5.11)]{202.8} can be bounded on a constant multiple of  $h_\delta/\tilde{p}$).
In this non-polytopal case, finite element spaces should be interpreted as spaces of piecewise mapped polynomials. Unfortunately, for our formulation of the Helmholtz equation as a first order system, in this setting where the aforementioned maps are non-polynomial, we do not know how to construct meaningful finite element subspaces $V_\pm^\delta$ of $V_\pm$. 
The problem is to satisfy the boundary condition $\vec{v}\cdot\vec{n}\pm i \eta^\delta=0$ for $(\eta^\delta,\vec{v}^\delta) \in V_\pm^\delta$.

For convex polygons and  $\partial\Omega=\Gamma_R$, in \cite[\S5.1.2]{202.8} it is shown that for $\tilde{\tria}^\delta$ being a quasi-uniform partition of $\Omega$ into $d$-simplices with mesh-size $\tilde{h}_\delta$, with some additional refinements in ${\mathcal O}(\tilde{h}_\delta)$-neighbourhoods of the vertices (which retain $\# \tilde{\tria}^\delta={\mathcal O}(\tilde{h}_\delta^{-d})$), $\kappa \tilde{h}_{\delta} \lesssim 1$, and $\tilde{p}^{-1} \log \kappa$ being sufficiently small, \eqref{MS} holds with a modified upper bound reading as $\frac{\tilde{h}_\delta \kappa+\tilde{h}_\delta^{1/(1+\eps)}}{\tilde{p}}$, where $\eps$ can be any positive number (see \cite[Remark~4.10]{202.8}).
We conclude that \eqref{MS} is valid when
$$
\tilde{h}_\delta \lesssim h_\delta^{1+\eps}.
$$

\begin{remark} The latter condition requires an arbitrary small increase in the algebraic order of complexity of $\tilde{\tria}^\delta$ compared to that of $\tria^\delta$.
In numerical experiments on general polygons we observe quasi-optimal approximations when taking $\tilde{\tria}^\delta=\tria^\delta$, even without the condition of $\tria^\delta$ being quasi-uniform.
\end{remark}


\begin{theorem} \label{thm:fine} \new{Let $\partial\Omega=\Gamma_R$, and let \eqref{bdrnorms} be valid. 
Let $U^\delta=\cS_p^{-1}(\tria^\delta)^{d+1}$, and let $V_\mp^\delta$ be as in \eqref{301}. 
Besides the conditions on $\tilde{\tria}^\delta$ and $\tilde{p}$ needed to guarantee \eqref{MS} {\rm (}on a convex polygon, sufficient is $(\log \kappa)/\tilde{p}$ sufficiently small, and, essentially, $\tilde{h}_\delta \lesssim h_\delta^{1+\eps}${\rm )}, assume that $\kappa h_\delta \lesssim 1$, $p^2/\tilde{p}$ is sufficiently small, and $\max_{K \in \tilde{\tria}^\delta} \diam(K) \lesssim h_\delta$.
Then $\gamma_\kappa^\delta \gtrsim 1$.}
\end{theorem}


\begin{proof} The proof consists of four steps. 
\new{{\bf (I)} In view of Proposition~\ref{proppie}, Lemma~\ref{lemma:extra}, and our assumption that $p^2/\tilde{p}$ is sufficiently small, it suffices to show that 
$$
\sup_{0 \neq \mathbbm{w}\in H^1_0(\Omega)^{d+1}}
\frac{
\inf_{\mathbbm{v}^\delta \in V_\mp^\delta} \|B_\kappa'\big[{B_\kappa'}^{-1} \mathbbm{w} -\mathbbm{v}^\delta\big]\|_U
}
{
\|\mathbbm{w}\|_{U_1}
}
\lesssim \frac{1}{\tilde{p}},
$$
where we recall that  $\|\cdot\|_{U_1}=\sqrt{\|\cdot\|^2_U+h_\delta\|\cdot\|_{H^1(\Omega)^{d+1}}^2}$.

Writing $\mathbbm{w} \in H^1_0(\Omega)^{d+1}$ as $(\phi,\vec{u})$, and denoting with $(\eta,\vec{v}) \in V_{\mp}$ the solution of
$$
B_\kappa' (\eta,\vec{v})=(-\kappa^{-1}\divv \vec{v}-\eta,\kappa^{-1} \nabla \eta-\vec{v})=(\phi,\vec{u}),
$$
by an application of \eqref{triangle}\footnote{Knowing that this inequality can be arbitrarily crude, see Remark~\ref{crude}, if we could avoid its use, then it might be possible to relax on the assumption~\eqref{MS}.} it suffices to show that
\be \label{p5}
\sup_{0 \neq (\phi,\vec{u}) \in U_1}\inf_{(\eta^\delta,\vec{v}^\delta) \in V^\delta_{\mp}} \frac{\|\eta-\eta^\delta\|_{1,\kappa}+\|\vec{v}-\vec{v}^\delta\|_{\divv,\kappa}}{\|(\phi,\vec{u})\|_{U_1}} \lesssim \new{\frac{1}{\tilde{p}}}.
\ee
Thanks to $\vec{u} \in H_0^1(\Omega)^d \subset H_0(\divv;\Omega)$, Lemma~\ref{lem:easy} shows that} $\eta \in H^1(\Omega)$ is the solution of $L_\kappa' \eta= \phi-\kappa^{-1} \divv \vec{u}$.
Condition~\eqref{MS} shows the existence of an $\eta^\delta \in  \cS_{\tilde{p}}^0(\tilde{\tria}^\delta)$ with
\be \label{p4}
 \|\eta-\eta^\delta\|_{1,\kappa} \lesssim \tfrac{h_\delta \kappa}{\tilde{p}} (\|\phi\|_{L_2(\Omega)}+\kappa^{-1}\|\divv \vec{u}\|_{L_2(\Omega)})
 \lesssim  \tfrac{1}{\tilde{p}} \|(\phi,\vec{u})\|_{U_1},
\ee
\new{by $\kappa h_\delta \lesssim 1$.}

\new{{\bf (II)}} To construct a similarly accurate approximation $\vec{v}^\delta \in RT_{\tilde{p}}(\tilde{\tria}^\delta)$ in $\|\cdot\|_{\divv,\kappa}$-norm to $\vec{v}=\kappa \nabla \eta -\vec{u}$, a complication is the boundary condition incorporated in the definition of $V^\delta_{\mp}$. The pair $(\eta^\delta,\vec{v}^\delta)$ 
has to satisfy $(\vec{v}^\delta\cdot \vec{n} \mp i \eta^\delta)|_{\partial\Omega}=0$.

\new{Knowing that the range of the normal trace operator restricted to $\RT_{\tilde{p}}(\tilde{\tria}^\delta)$ contains the range of the trace operator restricted to $\cS^0_{\tilde{p}}(\tilde{\tria}^\delta)$, we can pick some $\vec{z}^\delta\in \RT_{\tilde{p}}(\tilde{\tria}^\delta)$ with $(\vec{z}^\delta\cdot\vec{n} \mp i \eta^\delta)|_{\partial\Omega}=0$.

From $(\vec{v}\cdot\vec{n} \mp i \eta)|_{\partial\Omega}=0$, we have $(\vec{v}-\vec{z}^\delta)\cdot\vec{n}|_{\partial\Omega}=\pm i (\eta-\eta^\delta)|_{\partial\Omega}$, and so 
\be \label{p6}
\|(\vec{v}-\vec{z}^\delta)\cdot\vec{n}|_{\partial\Omega}\|_{-\frac12,\kappa}=
\|(\eta -\eta^\delta)|_{\partial\Omega}\|_{-\frac12,\kappa}  \lesssim \| (\eta -\eta^\delta)|_{\partial\Omega}\|_{\frac12,\kappa} \lesssim 
\|\eta -\eta^\delta\|_{1,\kappa}
\ee
by \eqref{bdrnorms}.
The definition of $\|\cdot\|_{-\frac12,\kappa}$ shows the existence of a $\vec{w} \in H(\divv;\Omega)$ with 
\be \label{p7}
\vec{w}\cdot\vec{n}|_{\partial\Omega}=(\vec{v}-\vec{z}^\delta)\cdot\vec{n}|_{\partial\Omega},\quad
\|\vec{w}\|_{\divv,\kappa} \lesssim \|(\vec{v}-\vec{z}^\delta)\cdot\vec{n}|_{\partial\Omega}\|_{-\frac12,\kappa}.
\ee
Using that $(\vec{v}-\vec{z}^\delta-\vec{w}) \in H_0(\divv;\Omega)$, with $Q^\delta\colon H_0(\divv;\Omega) \rightarrow \RT_{\tilde{p}}(\tilde{\tria}^\delta) \cap H_0(\divv;\Omega)$ being the projector constructed in \cite{70.991}
we select
$$
\vec{v}^\delta :=Q^\delta (\vec{v}-\vec{z}^\delta-\vec{w})+\vec{z}^\delta.
$$
and conclude that $(\eta^\delta, \vec{v}^\delta) \in V_{\mp}^\delta$.
What remains is to bound $\|\vec{v}-\vec{v}^\delta\|_{\divv,\kappa}$.}

\new{{\bf (III)}}  The projector $Q^\delta$ satisfies the commuting diagram property $\divv Q^\delta=\Pi^\delta \divv$, where $\Pi^\delta$ is the $L_2(\Omega)$-orthogonal projector onto $\cS^{-1}_{\tilde{p}}(\tilde{\tria}^\delta)$, and for $K \in \tilde{\tria}^\delta$, \new{$\vec{q} \in H_0(\divv;\Omega)$, it holds that}
\be \label{Voh}
\begin{split}
&\|(\identity-Q^\delta)\vec{q}\|_{L_2(K)}^2 \lesssim\\
&\hspace*{-0.5cm} 
\sum_{\{K'\in \tilde{\tria}^\delta\colon \overline{K} \cap \overline{K'} \neq \emptyset\}}
\hspace*{-2em}\min_{\mbox{}\hspace*{2em}\vec{q}_K \in \RT_{\tilde{p}}(K')} \hspace*{-0.3em}\|\vec{q}-\vec{q}_K\|_{L_2(K')}^2\!+\!\big(\tfrac{\diam(K)}{\tilde{p}}\big)^2 \|(\identity-\Pi^\delta)\divv \vec{q}\|^2_{L_2(K')}.
\end{split}
\ee
We infer that
\begin{align*}
\kappa^{-1}\|\divv(\vec{v}-\vec{v}^\delta)\|_{L_2(\Omega)}&=
\kappa^{-1}\|\divv(\vec{v}-Q^\delta (\vec{v}-\vec{z}^\delta-\vec{w})-\vec{z}^\delta)\|_{L_2(\Omega)}\\
&=
\kappa^{-1}\|(\identity-\Pi^\delta)\divv \vec{v}+(\Pi^\delta-\identity)\divv \vec{z}^\delta+\Pi^\delta\divv \vec{w}\|_{L_2(\Omega)}
\\
&=
\kappa^{-1}\|(\identity-\Pi^\delta)\divv \vec{v}+\Pi^\delta\divv \vec{w}\|_{L_2(\Omega)}.
\end{align*}
From \eqref{p6}-\eqref{p7} we know that 
$$
\kappa^{-1}\|\Pi^\delta\divv \vec{w}\|_{L_2(\Omega)} \leq \|\vec{w}\|_{\divv,\kappa} \lesssim \|\eta -\eta^\delta\|_{1,\kappa}.
$$
Using that $\kappa^{-1} \divv \vec{v}=-\eta-\phi$, $\cS^{0}_{\tilde{p}}(\tilde{\tria}^\delta) \subset \cS^{-1}_{\tilde{p}}(\tilde{\tria}^\delta)$, and $\max_{K \in \tilde{\tria}^\delta} \diam(K) \lesssim h_\delta$, we have
\be \label{ee1}
\begin{split}
\kappa^{-1}\|(\identity&-\Pi^\delta)\divv \vec{v}\|_{L_2(\Omega)} \leq \|(\identity-\Pi^\delta)\eta\|_{L_2(\Omega)}+\|(\identity-\Pi^\delta)\phi\|_{L_2(\Omega)}\\
&\leq \|\eta -\eta^\delta\|_{L_2(\Omega)}+\|(\identity-\Pi^\delta)\phi\|_{L_2(\Omega)}\\
&\lesssim \|\eta -\eta^\delta\|_{L_2(\Omega)}+\tfrac{h_\delta}{\tilde{p}}|\phi|_{H^1(\Omega)} \leq \|\eta -\eta^\delta\|_{1,\kappa}+\tfrac{1}{\tilde{p}} \|(\phi,\vec{u})\|_{U_1}.
\end{split}
\ee
\new{where a proof of the inequality ``$\lesssim$'' can be deduced from e.g.~\cite[Thm.~3.1]{18.52}.}

Using \eqref{p4}, we conclude that
\be \label{p10}
\kappa^{-1}\|\divv(\vec{v}-\vec{v}^\delta)\|_{L_2(\Omega)} \lesssim \tfrac{1}{\tilde{p}} \|(\phi,\vec{u})\|_{U_1}.
\ee

\new{{\bf (IV)}} What remains is to show such an upper bound for  $\|\vec{v}-\vec{v}^\delta\|_{L_2(\Omega)}$. 
We write $\vec{v}-\vec{v}^\delta=(\identity-Q^\delta)(\vec{v}-\vec{z}^\delta-\vec{w})+\vec{w}$.
It holds that
\be \label{ee2}
\|\vec{w}\|_{L_2(\Omega)} \leq \|\vec{w}\|_{\divv,\kappa} \lesssim \|\eta -\eta^\delta\|_{1,\kappa} \lesssim \tfrac{1}{\tilde{p}} \|(\phi,\vec{u})\|_{U_1}
\ee
again by \eqref{p4}. Using \eqref{Voh} we have
\begin{align} \nonumber
\|(&\identity-Q^\delta)(\vec{v}-\vec{z}^\delta-\vec{w})\|_{L_2(\Omega)}
\lesssim \tfrac{h_\delta}{\tilde{p}} \|(\identity-\Pi^\delta)\divv(\vec{v}-\vec{z}^\delta-\vec{w})\|_{L_2(\Omega)}\\  \nonumber
&\hspace*{4cm}+
\sqrt{\sum_{K \in \tilde{\tria}^\delta} \min_{\vec{q}_K \in \RT_{\tilde{p}}(K)}\| \vec{v}-\vec{z}^\delta-\vec{w}-\vec{q}_K\|^2_{L_2(K)}}\\  \nonumber
&=\tfrac{h_\delta}{\tilde{p}} \|(\identity-\Pi^\delta)\divv(\vec{v}-\vec{w})\|_{L_2(\Omega)}+
\sqrt{\sum_{K \in \tilde{\tria}^\delta} \min_{\vec{q}_K \in \RT_{\tilde{p}}(K)}\| \vec{v}-\vec{w}-\vec{q}_K\|^2_{L_2(K)}}\\ \label{ee3}
&\leq
\tfrac{h_\delta}{\tilde{p}} \|\divv \vec{w}\|_{L_2(\Omega)}+\|\vec{w}\|_{L_2(\Omega)}
+
\tfrac{h_\delta}{\tilde{p}} \|(\identity-\Pi^\delta)\divv \vec{v}\|_{L_2(\Omega)}+
\\ \nonumber
&\hspace*{5cm} \min_{\vec{q}^\delta \in \cS^{-1}_{\tilde{p}}(\tilde{\tria}^\delta)^d} \|\kappa^{-1} \nabla \eta -\vec{u}-\vec{q}^\delta\|_{L_2(\Omega)^d}\\ \label{ee4}
&\lesssim
\|\vec{w}\|_{\divv,\kappa} + \kappa^{-1}\|(\identity-\Pi^\delta)\divv \vec{v}\|_{L_2(\Omega)}+\kappa^{-1}|\eta-\eta^\delta|_{H^1(\Omega)}+\tfrac{h_\delta}{\tilde{p}}|\vec{u}|_{H^1(\Omega)}\\ \label{ee5}
&\lesssim  \|\eta -\eta^\delta\|_{1,\kappa}+\tfrac{1}{\tilde{p}} \|(\phi,\vec{u})\|_{U_1} \\ \label{ee6}
&\lesssim \tfrac{1}{\tilde{p}} \|(\phi,\vec{u})\|_{U_1},
\end{align}
where for \eqref{ee3} we have used that $\vec{v}=\kappa^{-1} \nabla \new{\eta}-\vec{u}$,
for  \eqref{ee4} that $\tfrac{h_\delta}{\tilde{p}} \leq h_\delta \lesssim \kappa^{-1}$ and $\nabla \cS^{0}_{\tilde{p}}(\tilde{\tria}^\delta) \subset \cS^{-1}_{\tilde{p}}(\tilde{\tria}^\delta)^d$,
for  \eqref{ee5} equations \eqref{ee2} and \eqref{ee1}, and
for  \eqref{ee6} equation \eqref{p4}.
We conclude that $\|\vec{v}-\vec{v}^\delta\|_{L_2(\Omega)} \lesssim \tfrac{1}{\tilde{p}} \|(\phi,\vec{u})\|_{U_1}$, which together with \eqref{p10} and \eqref{p4} shows \new{\eqref{p5}, and} thus completes the proof.
\end{proof}


\

\section{Comparison with standard Galerkin method} \label{sec:assessmentGalerkin}
The operator $L_\kappa$ that corresponds to the standard variational formulation of the Helmholtz equation is in $\Lis(H^1_{0,\Gamma_D}(\Omega),H^1_{0,\Gamma_D}(\Omega)')$ but it is, except for $\kappa$ small enough, not coercive. 
Nevertheless, in \cite{202.8,71.5} it was shown that for $\partial\Omega=\Gamma_R$, $\Omega$ a polygon, 
$\tria^\delta$ a quasi-uniform shape regular triangulation with mesh-size $h_\delta$, with additional refinements towards the vertices (which do not add to the order of complexity of $\tria^\delta$), and $\frac{\kappa h_\delta}{p}$ and $\frac{\log \kappa}{p}$ sufficiently small,  the Galerkin solution $\phi_{\text{Gal}}^\delta$ from
$$
X^\delta:=\cS_p^0(\tria^\delta) \cap H^1_{0,\Gamma_D}(\Omega)
$$
is quasi-best  w.r.t.~$\|\cdot\|_{1,\kappa}$. 
This means that the error in $\phi_{\text{Gal}}^\delta$ is at most a constant factor larger, only dependent on the shape regularity of $\tria^\delta$, than the error in the best approximation from $X^\delta$.
In other cases the Galerkin solution is observed \emph{not} to be quasi-best, a phenomenon  known as pollution. 

Notice that for $\vec{f}_2=0$ and $g_D=0$, for $\phi \in H^1_{0,\Gamma_D}(\Omega)$ being the solution of \eqref{2nd} it holds that $(\phi,\vec{u})$ with $\vec{u}=\frac{1}{\kappa}\nabla\phi$ is the solution of \eqref{eq:ultra-weak},
and that $\|\phi\|_{1,\kappa}=\|(\phi,\vec{u})\|_U$. So for trial spaces $X^\delta \subset H^1_{0,\Gamma_D}(\Omega)$ and $U^\delta \subset U$ that have comparable
orders it appropriate to compare the norm $\|\phi-\phi^\delta_{\rm Gal}\|_{1,\kappa}$ of the error in the Galerkin solution with the norm $\|(\phi,\vec{u})-(\phi^\delta,\vec{u}^\delta)\|_U$ of the error in the FOSLS solution.

Instead of comparing the errors in Galerkin and FOSLS approximations for some picks of the forcing term, we would like to compare how the errors in Galerkin and FOSLS approximations \emph{can}, and for some forcing terms \emph{will} deviate from the errors in the best approximations from the selected finite element spaces in $\|\cdot\|_{1,\kappa}$- and $\|\cdot\|_U$-norms respectively.
As we have seen in Theorem~\ref{thm1}, for the FOSLS approximation the attainable maximal deviation factor is given be the computable quantity $\frac{1}{\gamma_\kappa^\delta}$.
Our goal in the present section is to find a computable similar quantity for the Galerkin method.

\subsection{The pollution factor}
Knowing that $L_\kappa$ is boundedly invertible, $\|\cdot\|_{{\rm opt},\kappa}:=  \sup_{0 \neq \zeta \in H^1_{0,\Gamma_D}(\Omega)}\frac{|(L_\kappa' \cdot)(\zeta)|}{\|\zeta\|_{1,\kappa}}$ is a norm on $H^1_{0,\Gamma_D}(\Omega)$, known as the optimal test norm.
When $H^1_{0,\Gamma_D}(\Omega)'$ is equipped with the corresponding dual norm, and $H^1_{0,\Gamma_D}(\Omega)$ by $\|\cdot\|_{1,\kappa}$,
then $L_\kappa$ is an isometry.
Consequently, as follows from \cite[Remark 3.2]{249.99} or \cite[Sect.~2.1]{258.4}, with
$$
\gamma_{\kappa,{\rm Gal}}^\delta:=\inf_{0 \neq \phi^\delta \in X^\delta} \sup_{0 \neq \eta^\delta \in X^\delta} \frac{|(L_\kappa \phi^\delta)(\eta^\delta)|}{\|\phi^\delta\|_{1,\kappa}\|\eta^\delta\|_{\text{opt},\kappa}},
$$
it holds that
$$
\sup_{\phi \in H^1_{0,\Gamma_D}\setminus X^\delta} 
\frac{\|\phi-\phi_{\rm Gal}^\delta\|_{1,\kappa}}
{\inf_{\zeta^\delta \in X^\delta}\|\phi-\zeta^\delta\|_{1,\kappa}}=\frac{1}{\gamma_{\kappa,{\rm Gal}}^\delta}.\,\, \footnotemark
$$
\footnotetext{\cite[Thm.~2]{315.7} gives ``$\leq$''.}
We will call $1/\gamma_{\kappa,{\rm Gal}}^\delta$ the \emph{pollution factor of the Galerkin method}, and correspondingly, $1/\gamma_{\kappa}^\delta$ from Theorem~\ref{thm1} the \emph{pollution factor of the FOSLS method}.

Because of the presence of $\|\eta^\delta\|_{\text{opt},\kappa}$ in its definition, the constant $\gamma_{\kappa,{\rm Gal}}^\delta$ is, however, not computable.
For a suitable finite dimensional subspace $Y^\delta \subset H^1_{0,\Gamma_D}(\Omega)$, in the expression for $\gamma_{\kappa,{\rm Gal}}^\delta$ we will therefore replace
$\|\eta^\delta\|_{\text{opt},\kappa}$ by the computable quantity $\sup_{0 \neq \zeta^\delta \in Y^\delta} \frac{|(L_\kappa' \eta^\delta)(\zeta^\delta)|}{\|\zeta^\delta\|_{1,\kappa}}$. 

\begin{theorem} \label{thm:gal} Let either $\kappa h_{\delta,\min}$ be sufficiently large, or $\kappa h_{\delta,\max} \lesssim 1$.\footnote{Together both cases cover all possible combinations of quasi-uniform partitions and $\kappa \geq \kappa_0>0$.} Then
with $Y^\delta:=H^1_{0,\Gamma_D}(\Omega) \cap \prod_{K \in \tria^\delta} P_{p+d+1}(K)$, it holds that
$$
\|\cdot\|_{{\rm opt},\kappa} \lesssim \sup_{0 \neq \zeta^\delta \in Y^\delta} \frac{|(L_\kappa' \cdot)(\zeta^\delta)|}{\|\zeta^\delta\|_{1,\kappa}} \,\leq \|\cdot\|_{{\rm opt},\kappa} \quad \text{on } X^\delta,
$$
and so
$$
\gamma_{\kappa,{\rm Gal}}^\delta \leq\,\, \hat{\gamma}_{\kappa,{\rm Gal}}^{\delta}:=\inf_{0 \neq \phi^\delta \in X^\delta} \sup_{0 \neq \eta^\delta \in X^\delta} \frac{|(L_\kappa \phi^\delta)(\eta^\delta)|}{\|\phi^\delta\|_{1,\kappa}
\sup_{0 \neq \zeta^\delta \in Y^\delta} \frac{|(L_\kappa' \eta^\delta)(\zeta^\delta)|}{\|\zeta^\delta\|_{1,\kappa}}} \,\,\lesssim \gamma_{\kappa,{\rm Gal}}^\delta,\footnotemark
$$
\new{only dependent on $d$, the shape regularity of $\tria^\delta$, and here also on $p$.}
\footnotetext{So the true pollution factor of the Galerkin method $1/\gamma_{\kappa,{\rm Gal}}^\delta$ is greater than or equal to $1/\hat{\gamma}_{\kappa,{\rm Gal}}^{\delta}$.
Numerical results will indicate that $\hat{\gamma}_{\kappa,{\rm Gal}}^{\delta}$ is very close to $\gamma_{\kappa,{\rm Gal}}^{\delta}$.}
\end{theorem}

\begin{proof} {\bf (I)} The case of $\kappa h_{\delta,\min}$ being sufficiently large. 

\new{As a consequence of \eqref{81}}, for $\tilde{Y}^\delta := \prod_{K \in \tria^\delta} P_{p+d+1}(K) \cap H^1_0(K) \subset H^1_{0,\Gamma_D}(\Omega)$, it holds that
$$
\inf_{0 \neq \eta^\delta \in X^\delta} \sup_{0 \neq \zeta^\delta \in \tilde{Y}^\delta} \frac{\big|\langle \eta^\delta,\zeta^\delta\rangle_{L_2(\Omega)}\big|}{\|\eta^\delta\|_{L_2(\Omega)}\|\zeta^\delta\|_{L_2(\Omega)}}\gtrsim 1.
$$
Application of inverse inequalities for $\eta^\delta \in X^\delta$ and $\zeta^\delta \in \tilde{Y}^\delta$ gives
$$
\big|(L'_\kappa \eta^\delta)(\zeta^\delta)-\langle \eta^\delta,\zeta^\delta\rangle_{L_2(\Omega)}\big| \lesssim (\kappa h_{\delta,\min})^{-2} \|\eta^\delta\|_{L_2(\Omega)}\|\zeta^\delta\|_{L_2(\Omega)},
$$
and so for $\kappa h_{\delta,\min}$ sufficiently large,
\be \label{q3}
\inf_{0 \neq \eta^\delta \in X^\delta} \sup_{0 \neq \zeta^\delta \in \tilde{Y}^\delta} \frac{\big|(L'_\kappa \eta^\delta)(\zeta^\delta)\big|}{\|\eta^\delta\|_{L_2(\Omega)}\|\zeta^\delta\|_{L_2(\Omega)}}\gtrsim 1.
\ee

An application of the trace inequality in the form $\|\cdot\|^2_{L_2(\partial\Omega)} \lesssim \kappa \|\cdot\|^2_{L_2(\Omega)} + \kappa^{-1} |\cdot|^2_{H^1(\Omega)}$ shows that for $\eta,\zeta \in H^1_{0,\Gamma_D}(\Omega)$, $|(L_\kappa' \eta)(\zeta)| \lesssim \|\eta\|_{1,\kappa} \|\zeta\|_{1,\kappa}$, and
so $\|\cdot\|_{{\rm opt},\kappa} \lesssim \|\cdot\|_{1,\kappa}$.
Together with an inverse inequality,
thanks to $\kappa h_{\delta,\min} \gtrsim 1$ it gives that $\|\eta^\delta\|_{{\rm opt},\kappa} \lesssim \|\eta^\delta\|_{1,\kappa} \lesssim \|\eta^\delta\|_{L_2(\Omega)}$ ($\eta^\delta \in X^\delta$).
Again by $\kappa h_{\delta,\min} \gtrsim 1$, an inverse inequality shows that $\|\zeta^\delta\|_{1,\kappa} \lesssim \|\zeta^\delta\|_{L_2(\Omega)}$ ($\zeta^\delta \in \tilde{Y}^\delta$).  Together with \eqref{q3} the latter two estimates complete the proof of Theorem~\ref{thm:gal} for the case that $\kappa h_{\delta,\min}$ is sufficiently large.
\smallskip

{\bf (II)} The case $\kappa h_{\delta,\max} \lesssim 1$.

Let both $\Gamma_D$ and  $\Gamma_R$ be the union of (closed) facets $e$ of $\tria^\delta$.
Let $\cF(\tria^\delta)$ denote the set of facets $e$ of $\tria^\delta$ that are not on $\Gamma_D$.
For $K \in \tria^\delta$, we set the patches $\omega_{K,0}(\tria^\delta):=K$, and $\omega_{K,i+1}(\tria^\delta):=\cup\{K' \in \tria^\delta\colon K' \cap \omega_{K,i}(\tria^\delta) \neq \emptyset\}$. Let $\underline{h}_\delta$ be the piecewise constant function on $\Omega$ defined by $\underline{h}_\delta|_K:=|K|^{1/d}$.

It suffices to construct a `Fortin' interpolator $Q^\delta\colon H^1_{0,\Gamma_D}(\Omega) \rightarrow Y^\delta$ with
\be \label{q5}
\sup_{0\neq \eta \in H^1_{0,\Gamma_D}(\Omega)} \frac{\|Q^\delta \eta\|_{1,\kappa}}{\|\eta\|_{1,\kappa}}\lesssim 1,
\quad (L_\kappa' X^\delta)(\ran (\identity -Q^\delta))=0.
\ee
Since for $\phi^\delta \in X^\delta$ and $\eta \in H^1_{0,\Gamma_D}(\Omega)$,
$$
(L_\kappa' \phi^\delta)(\eta)=\sum_{K \in \tria^\delta} \big\{\int_K (-\kappa^{-2} \Delta\phi^\delta-\phi^\delta) \overline{\eta}\,dx+\new{\kappa^{-2}}\int_{\partial K} \tfrac{\partial \phi^\delta}{\partial\vec{n}} \overline{\eta}\big\} \mp \tfrac{i}{\kappa} \int_{\Gamma_R} \phi^\delta \overline{\eta}\,ds,
$$
the second condition holds when for all $K \in \tria^\delta$ and $e \in \cF(\tria^\delta)$
\be \label{q4}
\ran (\identity - Q^\delta)|_{K} \perp_{L_2(K)} P_p(K),\quad \ran (\identity - Q^\delta)|_{e} \perp_{L_2(e)} P_p(e).
\ee

Let $Q^\delta_{-2}\colon H^1_{0,\Gamma_D}(\Omega) \rightarrow S^0_1(\tria^\delta) \cap H^1_{0,\Gamma_D}(\Omega)$ denote the familiar Scott-Zhang interpolator (\cite{247.2}). It satisfies 
$$
\|\underline{h}_\delta^{-1}(\identity-Q^\delta_{-2})v\|_{L_2(K)} + |Q^\delta_{-2} v|_{H^1(K)} \lesssim |v|_{H^1(\omega_{K,1}(\tria^\delta))} \quad(v \in H^1_{0,\Gamma_D}(\Omega) ).
$$
Thanks to $\kappa h_{\max}^\delta \lesssim 1$, \new{this estimate implies that for $v \in H^1_{0,\Gamma_D}(\Omega)$,
$$
\|Q_{-2}^\delta v\|_{1,\kappa} \lesssim \|v\|_{L_2(\Omega)}+ \|(\identity -Q_{-2}^\delta)v\|_{L_2(\Omega)}+\kappa^{-1}|Q_{-2}^\delta v|_{H^1(\Omega)} \lesssim \|v\|_{1,\kappa},
$$
so that $Q_{-2}^\delta$ satisfies the first condition in \eqref{q5}.} In two steps we will construct a modified interpolator that also satisfies \eqref{q4}.

On a facet $\hat{e}$ of a reference $d$-simplex $\hat{K}$, let $b_{\hat{e}}$ denote the $d$-fold product of its barycentric coordinates.
From $\int_{\hat{e}} b_{\hat{e}} |q|^2\,ds \eqsim \int_{\hat{e}} |q|^2\,ds$ ($q \in P_p(\hat{e})$), and $b_{\hat{e}} P_p(\hat{e})=P_{p+d}(\hat{e}) \cap H^1_0(\hat{e})$, one infers that there exist bases $\{\hat{\tilde{\psi}}_i\}_i$ and $\{\hat{\ell}_i\}_i$ of $P_{p+d}(\hat{e}) \cap H^1_0(\hat{e})$ and $P_p(\hat{e})$ that are $L_2(\hat{e})$-biorthogonal. Let $\hat{\psi}_i$ be an extension of $\hat{\tilde{\psi}}_i$ to a function in $P_{p+d}(\hat{K}) \cap H^1_{0,\partial \hat{K}\setminus {\rm int}(\hat{e})}(\hat{K})$.

By using affine bijections between $\hat{K}$ and $K \in \tria^\delta$, for each $e \in \cF(\tria^\delta)$ we lift $\{\hat{\ell}_i\}_i$ to a collection $\{\ell_{e,i}\}_i$ that spans $P_p(e)$, and  lift $\{\hat{\psi}_i\}_i$ to a collection $\{\psi_{e,i}\}_i$ of functions on the union of the two (or one) simplices in $\tria^\delta$ of which $e$ is a facet.
We set 
$$
Q^\delta_{-1} v:=Q^\delta_{-2} v+\sum_{e \in  \cF(\tria^\delta)}\sum_i \frac{\langle v-Q^\delta_{-2} v,\ell_{e,i}\rangle_{L_2(e)}}{\langle \psi_{e,i},\ell_{e,i}\rangle_{L_2(e)}}\psi_{e,i}
\quad (v \in H^1_{0,\Gamma_D}(\Omega)).
$$
From $\langle \psi_{e,i},\ell_{e,j}\rangle_{L_2(e)}=0$ when $i \neq j$, it follows that
$\ran (\identity - Q^\delta_{-1})|_{e} \perp_{L_2(e)} P_p(e)$.
Standard homogeneity arguments and the use of the trace inequality show that 
$$
\|\underline{h}_\delta^{-1}(\identity-Q^\delta_{-1})v\|_{L_2(K)} + |Q^\delta_{-1} v|_{H^1(K)} \lesssim |v|_{H^1(\omega_{K,2}(\tria^\delta))} \quad(v \in H^1_{0,\Gamma_D}(\Omega) ).
$$

Let $b_{\hat{K}}$ denote the $(d+1)$-fold product of the barycentric coordinates of $\hat{K}$.
From $\int_{\hat{K}} b_{\hat{K}} |q|^2\,dz \eqsim \int_{\hat{K}} |q|^2\,dx$ ($q \in P_p(\hat{K})$), and $b_{\hat{K}} P_p(\hat{K})=P_{p+d+1}(\hat{K}) \cap H^1_0(\hat{K})$, one infers that there exist bases $\{\hat{\phi}_k\}_k$ and $\{\hat{q}_k\}_k$ of $P_{p+d+1}(\hat{K}) \cap H^1_0(\hat{K})$ and $P_d(\hat{K})$ that are $L_2(\hat{K})$-biorthogonal. 

Again using the affine bijections between $\hat{K}$ and $K \in \tria^\delta$, for each $K \in \tria^\delta$ 
we lift $\{\hat{\phi}_k\}_k$ and $\{\hat{q}_k\}_k$ to collections $\{\phi_{K,k}\}_k$ and $\{q_{K,k}\}_k$ that span $P_{p+d+1}(K) \cap H^1_0(K)$ and $P_p(K)$, respectively.
We set 
$$
Q^\delta v:=Q^\delta_{-1} v+\sum_{K \in  \tria^\delta}\sum_k \frac{\langle v-Q^\delta_{-1} v,q_{K,k}\rangle_{L_2(K)}}{\langle \phi_{K,k},q_{K,k}\rangle_{L_2(K)}}\phi_{K,k}
\quad (v \in H^1_{0,\Gamma_D}(\Omega).
$$
It satisfies \eqref{q4}, as well as 
$$
\|\underline{h}_\delta^{-1}(\identity-Q^\delta)v\|_{L_2(K)} + |Q^\delta v|_{H^1(K)} \lesssim |v|_{H^1(\omega_{K,2}(\tria^\delta))} \quad(v \in H^1_{0,\Gamma_D}(\Omega) ),
$$
\new{which, as we have seen, for $\kappa h_{\max}^\delta \lesssim 1$ } implies the first condition in \eqref{q5} and so completes the proof.
\end{proof}

\subsection{Computation of the pollution factors $1/\gamma_\kappa^\delta$ and $1/\hat{\gamma}_{\kappa,{\rm Gal}}^\delta$ of the FOSLS and Galerkin method}
Let $\{\varphi_1,\varphi_2,\ldots\}$, $\{\psi_1,\psi_2,\ldots\}$ be bases for (finite dimensional) $U^\delta$ and $V_{\mp}^\delta$, respectively.
Define the matrices ${\bf M}^U$, ${\bf M}^{V_{\mp}}$, ${\bf B}$ by
${\bf M}^U_{i j}=\langle \varphi_j,\varphi_i \rangle_U$, 
${\bf M}^{V_{\mp}}_{i j}=\langle B_\kappa' \psi_j,B_\kappa' \psi_i \rangle_U$,
${\bf B}_{i j}=\langle \varphi_j , B_\kappa' \psi_i \rangle_U$.
Then one computes the constant $(\gamma^\delta_\kappa)^2$ as the smallest generalized eigenvalue $\lambda$ of the pencil $\big({\bf B}^H ({\bf M}^{V_{\mp}})^{-1} {\bf B},{\bf M}^U\big)$ (i.e.,
${\bf B}^H ({\bf M}^{V_{\mp}})^{-1} {\bf B} {\bf x}=\lambda {\bf M}^U {\bf x}$ for some vector ${\bf x} \neq 0$).

Similarly, let $\{\varphi_1,\varphi_2,\ldots\}$ and $\{\psi_1,\psi_2,\ldots\}$ be bases for $X^\delta$ and $Y^\delta$.
Define matrices ${\bf M}^X$, ${\bf M}^Y$, ${\bf \tilde{L}}$, and ${\bf L}$ by
${\bf M}^X_{i j}=\langle \varphi_j,\varphi_i \rangle_{1,\kappa}$, 
${\bf M}^Y_{i j}=\langle  \psi_j, \psi_i \rangle_{1,\kappa}$,
${\bf \tilde{L}}_{i j}=(L_\kappa \psi_j)(\varphi_i)$, and
${\bf L}_{i j}=(L_\kappa \varphi_j)(\varphi_i)$.
Then one computes the constant $(\hat{\gamma}_{\kappa,{\rm Gal}}^\delta)^2$ as the smallest generalized eigenvalue of the pencil $\big({\bf L}^H ({\bf \tilde{L}} ({\bf M}^Y)^{-1} {\bf \tilde{L}}^H)^{-1} {\bf L},{\bf M}^X\big)$.

\section{Boosted FOSLS and a posteriori error estimation} \label{sec:boosted}
Recall that our ultra-weak FOSLS solution, given by 
$$
\mathbbm{u}^\delta:=\argmin_{\mathbbm{w}^\delta \in U^\delta} \sup_{0 \neq \mathbbm{z}^\delta \in V^\delta_{\mp}} \frac{|(q- B_\kappa \mathbbm{w}^\delta)(\mathbbm{z}^\delta)|}{\|\mathbbm{z}^\delta\|_{V_{\mp,\kappa}}},
$$ 
can be computed as the second component of $(\mathbbm{v}^\delta,\mathbbm{u}^{\delta}) \in V_{\mp}^\delta \times U^\delta$ that satisfies
\be \label{300}
\left\{\hspace*{-0.5em}
\begin{array}{lcll}
\langle B_\kappa'\mathbbm{v}^\delta, B_\kappa'\undertilde{\mathbbm{v}}^\delta\rangle_U +
\langle \mathbbm{u}^{\delta}, B_\kappa'\undertilde{\mathbbm{v}}^\delta\rangle_U 
& \!\!= \!\!& q(\undertilde{\mathbbm{v}}^\delta) & (\undertilde{\mathbbm{v}}^\delta \in V_{\mp}^\delta),\\
\langle B_\kappa'\mathbbm{v}^\delta,\undertilde{\mathbbm{u}}^\delta\rangle_U
& \!\!=\!\! & 0 & (\undertilde{\mathbbm{u}}^\delta \in U^\delta).
\end{array}
\right.\hspace*{-1.2em}
\ee

Apart from its use to compute $\mathbbm{u}^\delta$, we will demonstrate that $\mathbbm{v}^\delta$ can be employed to construct an improved approximation.
The next theorem shows that $\mathbbm{u}^\delta+B_\kappa'\mathbbm{v}^\delta$ is always more accurate than $\mathbbm{u}^\delta$ (more precisely, never less accurate), and, assuming $\gamma_\kappa^\delta \gtrsim 1$, that it is a quasi-best approximation from $U^\delta+ (B_\kappa' V^\delta_{\new{\mp}} \cap (U^\delta)^{\perp})$.

\begin{theorem} It holds that
\be \label{201}
\|\mathbbm{u}-(\mathbbm{u}^\delta+B_\kappa'\mathbbm{v}^\delta)\|_U^2=\|\mathbbm{u}-\mathbbm{u}^\delta\|_U^2-\|B_\kappa'v^\delta\|_U^2,
\ee
and
$$
\|\mathbbm{u}-(\mathbbm{u}^\delta+B_\kappa'\mathbbm{v}^\delta)\|_U \leq \frac{1}{\gamma_\kappa^\delta} \,\,\inf_{\mathbbm{w}^\delta \in U^\delta+ (B_\kappa' V^\delta_{\new{\mp}} \cap (U^\delta)^{\perp})} \|\mathbbm{u}-\mathbbm{w}^\delta\|_U.
$$
\end{theorem}

\begin{proof} Substituting $q(\undertilde{\mathbbm{v}}^\delta)=\langle \mathbbm{u}, B_\kappa'\undertilde{\mathbbm{v}}^\delta\rangle_U$ in \eqref{300}
shows that $B_\kappa'\mathbbm{v}^\delta$ is the orthogonal projection of $\mathbbm{u}-\mathbbm{u}^{\delta}$ onto $B_\kappa' V_{\mp}^\delta$, which completes the proof of \eqref{201}.

In Remark~\ref{remmie} we have seen that $P_1^\delta:=\mathbbm{u} \mapsto \mathbbm{u}^\delta$ is a projector onto $U^\delta$ with $\frac{1}{\gamma_\kappa^\delta}=\|P_1^\delta\|_{\cL(U,U)}=\|\identity - P_1^\delta\|_{\cL(U,U)}$. 
From \eqref{300} one infers that $P_2^\delta:=\mathbbm{u} \mapsto \mathbbm{u}^\delta+B_\kappa'\mathbbm{v}^\delta$ is a projector onto $U^\delta+ (B_\kappa' V^\delta_{\new{\mp}} \cap (U^\delta)^{\perp})$.
Equation \eqref{201} shows that $\|\identity - P_2^\delta\|_{\cL(U,U)}\leq \|\identity - P_1^\delta\|_{\cL(U,U)}=\frac{1}{\gamma_\kappa^\delta}$.
From
$$
\|(\identity - P_2^\delta)\mathbbm{u}\|_U = \inf_{\mathbbm{w}^\delta \in U^\delta+ (B_\kappa' V^\delta_{\new{\mp}} \cap (U^\delta)^{\perp})}\|(\identity - P_2^\delta)(\mathbbm{u}-\mathbbm{w}^\delta)\|_U,
$$
one concludes the proof.
\end{proof}

In numerical experiments we will see that from the point on where the number of degrees of freedom per wavelength exceeds $1$, 
the boosted approximation
$$
(\phi_{\rm bst}^\delta,\vec{u}_{\rm bst}^\delta):=\mathbbm{u}^\delta+B_\kappa'\mathbbm{v}^\delta
$$
is an increasingly more accurate approximation to $\mathbbm{u}$ than $\mathbbm{u}^\delta$. Moreover, having solved the saddle point system \eqref{201}, one gets this improved approximation for free. On the other hand, the representation of $\mathbbm{u}^\delta+B_\kappa'\mathbbm{v}^\delta$ requires more storage than that of $\mathbbm{u}^\delta$. Furthermore, because of the unusual `trial space' $U^\delta+ (B_\kappa' V^\delta_{\new{\mp}} \cap (U^\delta)^{\perp})$, it is not so clear how to compare this boosted approximation with a Galerkin approximation using a standard finite element space.

\begin{remark}[$L L^*$-method] Keeping $V^\delta_{\new{\mp}}$, but replacing $U^\delta$ by $\{0\}$, the boosted FOSLS approximation reduces to finding $\mathbbm{v}^\delta \in V^\delta_{\new{\mp}}$ that solves
$\langle B_\kappa'\mathbbm{v}^\delta, B_\kappa'\undertilde{\mathbbm{v}}^\delta\rangle_U= q(\undertilde{\mathbbm{v}}^\delta)$  for all $\undertilde{\mathbbm{v}}^\delta \in V_{\mp}^\delta$. This is an example of a so-called $L L^*$-method (see \cite{35.9356}). 
It only requires solving a Hermitian positive definite system.
It holds that $\|\mathbbm{u}-B_\kappa'\mathbbm{v}^\delta\|_U=
\|B_\kappa'^{-1} \mathbbm{u}-\mathbbm{v}^\delta\|_{V_{\mp,\kappa}}=
\min_{\tilde{\mathbbm{v}}^\delta \in V^\delta_{\new{\mp}}} \|B_\kappa'^{-1} \mathbbm{u}-\tilde{\mathbbm{v}}^\delta\|_{V_{\mp,\kappa}}$.
Even for fixed $\kappa$, to estimate this best approximation error the question of the regularity of $B_\kappa'^{-1} \mathbbm{u}$ enters. 
We refer to \cite{170.5} for a further discussion of this issue.

In our numerical experiments, the boosted FOSLS approximation was much more accurate than this $L L^*$-approximation.
\end{remark}

The boosted approximation $\mathbbm{u}^\delta+B_\kappa' \mathbbm{v}^\delta$ being much more accurate than $\mathbbm{u}^\delta$ is equivalent
to $B_\kappa' \mathbbm{v}^\delta$ being close to the error $\mathbbm{u}-\mathbbm{u}^\delta$. In particular, if, for some constant $\lambda<1$, $\|\mathbbm{u} -(\mathbbm{u}^\delta+B_\kappa' \mathbbm{v}^\delta)\|_U \leq \lambda 
\|\mathbbm{u} -\mathbbm{u}^\delta\|_U$, then \eqref{201} shows that 
$$
\|B_\kappa' \mathbbm{v}^\delta\|_U \leq \|\mathbbm{u} -\mathbbm{u}^\delta\|_U \leq \tfrac{1}{\sqrt{1-\lambda^2}}\|B_\kappa' \mathbbm{v}^\delta\|_U.
$$
A value of $\lambda<1$ can only be expected when the number of degrees of freedom in $V_\mp^\delta$ per wavelength start to exceed $1$.
This is confirmed in numerical experiments, in which it is also observed that $\lambda \downarrow 0$ for a mesh-size tending to zero.
 
 Furthermore, we will employ local $U$-norms of $B_\kappa'\mathbbm{v}^\delta$ as local error indicators to drive an adaptive refinement routine.

\begin{remark}
In Theorems~\ref{thm:coarse} and \ref{thm:fine}, we showed inf-sup stability $\gamma_\kappa^\delta$ without constructing a Fortin interpolator. If one has a Fortin interpolator $\Pi^\delta$ available, then following \cite{35.93556} one shows that
$$
\|\mathbbm{u} -\mathbbm{u}^\delta\|_U \leq \|B_\kappa' \Pi^\delta {B_\kappa'}^{-1}\|_{\cL(U,U)} \|B_\kappa' \mathbbm{v}^\delta\|_U+\osc^\delta(q),
$$
with the data-oscillation term $\osc^\delta(q):=\|B_\kappa^{-1}(\identity-\Pi^\delta)'q\|_U$.
\end{remark}

\

\section{Numerical results} \label{sec:numerics}
\subsection{Pollution factors}
For our first experiment, we take $\kappa=100$, $\Omega=(0,1)^2$, $\partial\Omega=\Gamma_R$, $\tria^\delta$ the uniform criss-cross triangulation into isosceles right triangles with longest edge of size $h_\delta$, trial space
$U^\delta=\cS_p^{0}(\tria^\delta) \times \cS_p^{0}(\tria^\delta)^{2}$, test space
$V_\mp^\delta=(\cS_{\new{\tilde{p}}}^{0}(\tria^\delta) \times \RT_{\new{\tilde{p}}}(\tria^\delta)) \cap V_\mp$ for $\tilde{p}=p+2$, and for the Galerkin method,
$X^\delta=\cS_p^{0}(\tria^\delta)$, and, merely for the computation of the approximate pollution factor $1/\hat{\gamma}_{\kappa,{\rm Gal}}^\delta$ of the Galerkin method, $Y^\delta=\cS_{p+3}^{0}(\tria^\delta)$.
In Figure~\ref{fig:pollution}, for $p \in \{1,2,3,4\}$ we compare the pollution factors $1/\gamma_\kappa^\delta$ (FOSLS) and $1/\hat{\gamma}_{\kappa,{\rm Gal}}^\delta$ (Galerkin) as function of the mesh-size. The factors $1/\hat{\gamma}_{\kappa,{\rm Gal}}^\delta$ hardly increased when we increased the polynomial  degree of $Y^\delta$,  and so they are expected to be very close to the true pollution factors $1/\gamma_{\kappa,{\rm Gal}}^\delta$.
\begin{figure}[h!]
\begin{subfigure}{.5\textwidth}
\centering
\includegraphics[width=\linewidth]{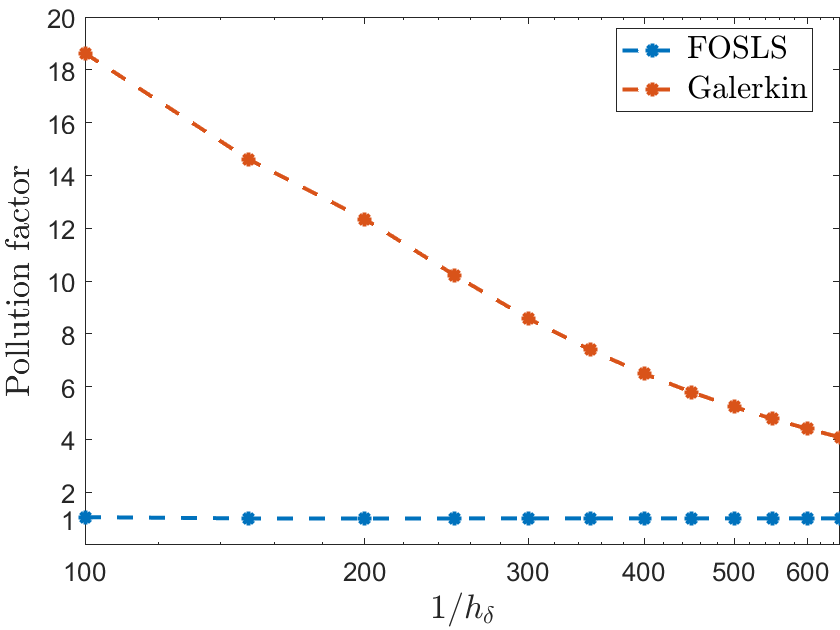}
\caption{$p=1$}
\end{subfigure}\hspace*{0cm}%
\begin{subfigure}{.5\textwidth}
\centering
\includegraphics[width=\linewidth]{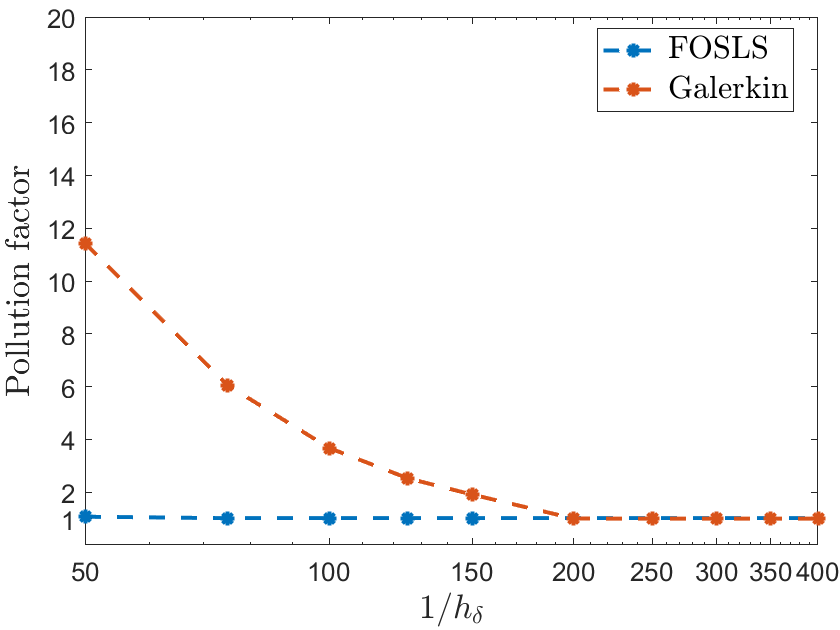}
\caption{$p=2$}
\end{subfigure}
\begin{subfigure}{.5\textwidth}
\centering
\includegraphics[width=\linewidth]{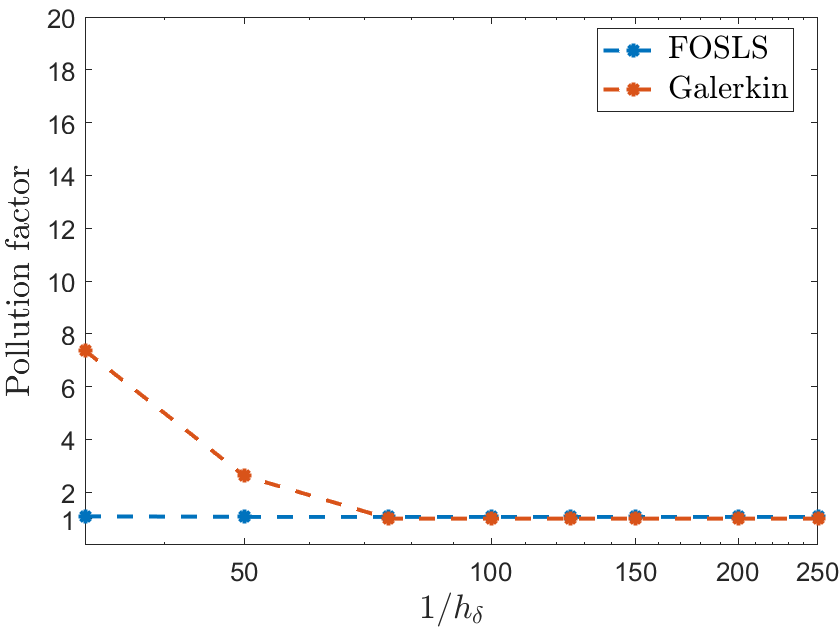}
\caption{$p=3$}
\end{subfigure}\hspace*{0cm}%
\begin{subfigure}{.5\textwidth}
\centering
\includegraphics[width=\linewidth]{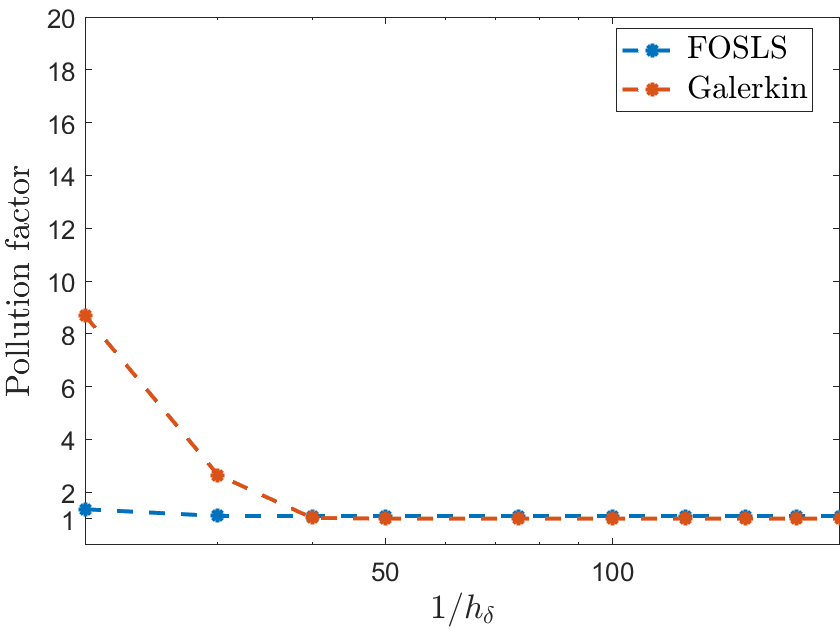}
\caption{$p=4$}
\end{subfigure}
\caption{Comparison of pollution factors $1/\gamma_\kappa^\delta$ (FOSLS) and $1/\hat{\gamma}_{\kappa,{\rm Gal}}^\delta$ (Galerkin) for $p \in \{1,2,3,4\}$ and $\kappa=100$.}
\label{fig:pollution}
\end{figure}
In all cases, the pollution factor for the FOSL method is very close to $1$, meaning that the FOSLS solution is very close to the best approximation from $U^\delta$ w.r.t.~the $U$-norm. Concerning the Galerkin method, the results show that in order to get a pollution factor below, say, $4$, the number of points required per wavelength $\frac{2\pi p}{\kappa h_\delta}$
is $41.6$, $12.1$, $8.4$, $7.0$ for $p=1,2,3,4$, respectively.
These numbers are known to further increase when $\kappa \rightarrow \infty$, unless $p$ is simultaneously increased such that $\frac{\log \kappa}{p}$ is kept sufficiently small.
The pollution factor for the FOSLS method can be uniformly controlled when \new{$\max(\log \kappa,p^2)/\tilde{p}$ is sufficiently small.}
\subsection{Comparison of errors for a plane wave solution}
As in the previous subsection, in this experiment we take $\kappa=100$, $\Omega=(0,1)^2$, $\partial\Omega=\Gamma_R$, $\tria^\delta$ the uniform triangulation with mesh-size $h_\delta$, and, here for $p=1$,
$U^\delta=\cS_p^{0}(\tria^\delta) \times \cS_p^{0}(\tria^\delta)^{2}$,
$V_\mp^\delta=(\cS_{p+2}^{0}(\tria^\delta) \times \RT_{p+2}(\tria^\delta)) \cap V_\mp$, and $X^\delta=\cS_p^{0}(\tria^\delta)$.

For $\vec{r}=(\cos(\pi/3),\sin(\pi/3))$, we prescribe the plane wave solution 
$\phi(\vec{x})=\phi_{\kappa \vec{r}}(\vec{x}):=e^{-i \kappa \vec{r}\cdot\vec{x}}$, and so $\vec{u}:=\frac{1}{\kappa} \nabla \phi=-i \phi \vec{r}$, and choose the data for the FOSLS and Galerkin methods correspondingly.

With $(\phi^\delta,\vec{u}^\delta)$ and $\phi_{\rm Gal}^\delta$ denoting the FOSLS and Galerkin solution, and with $(\phi_{\rm bst}^\delta,\vec{u}_{\rm bst}^\delta)$ denoting the boosted FOSLS solution, for a range of $h_\delta$ we compare the errors 
$\|\phi-\phi^\delta\|_{L_2(\Omega)}$ and $\|\phi-\phi_{\rm Gal}^\delta\|_{L_2(\Omega)}$ with 
$\inf_{\tilde{\phi}^\delta \in \cS_1^{0}(\tria^\delta)} \|\phi-\tilde{\phi}^\delta\|_{L_2(\Omega)}$, and additionally give
$\|\phi-\phi_{\rm bst}^\delta\|_{L_2(\Omega)}$.

Similarly, we compare
$\|(\phi,\vec{u})-(\phi^\delta,\vec{u}^\delta)\|_{U}$ with
$\inf_{(\tilde{\phi}^\delta,\tilde{\vec{u}}^\delta) \in U^\delta} \|(\phi,\vec{u})-(\tilde{\phi}^\delta,\tilde{\vec{u}}^\delta)\|_{U}$,
and
$\|\phi-\phi_{\rm Gal}^\delta\|_{1,\kappa}$ with
$\inf_{\tilde{\phi}^\delta \in X^\delta}\|\phi-\tilde{\phi}^\delta\|_{1,\kappa}$.
Because of $\|\cdot\|_{1,\kappa}=\|(\cdot,\frac{1}{\kappa}\divvv)\|_U$, these four errors can also be mutually compared.
We give $\|(\phi,\vec{u})-(\phi_{\rm bst}^\delta,\vec{u}_{\rm bst}^\delta)\|_{U}$, and the effectivity index
$\|(\phi_{\rm bst}^\delta,\vec{u}_{\rm bst}^\delta)-(\phi^\delta,\vec{u}^\delta)\|_{U}/ \|(\phi,\vec{u})-(\phi^\delta,\vec{u}^\delta)\|_{U}$
of the computable estimator $(\phi_{\rm bst}^\delta,\vec{u}_{\rm bst}^\delta)-(\phi^\delta,\vec{u}^\delta)$ of the error in the FOSLS solution.

The results given in Figure~\ref{fig:errorsunitsquare}
\begin{figure}[h!]
\begin{subfigure}{.5\textwidth}
\centering
\includegraphics[width=\linewidth]{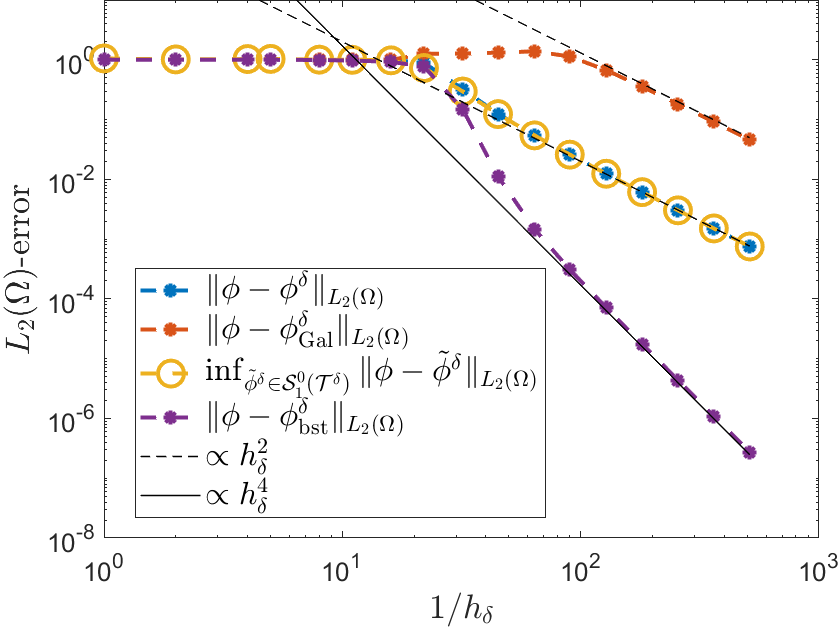}
\end{subfigure}\hspace*{0cm}%
\begin{subfigure}{.5\textwidth}
\centering
\includegraphics[width=\linewidth]{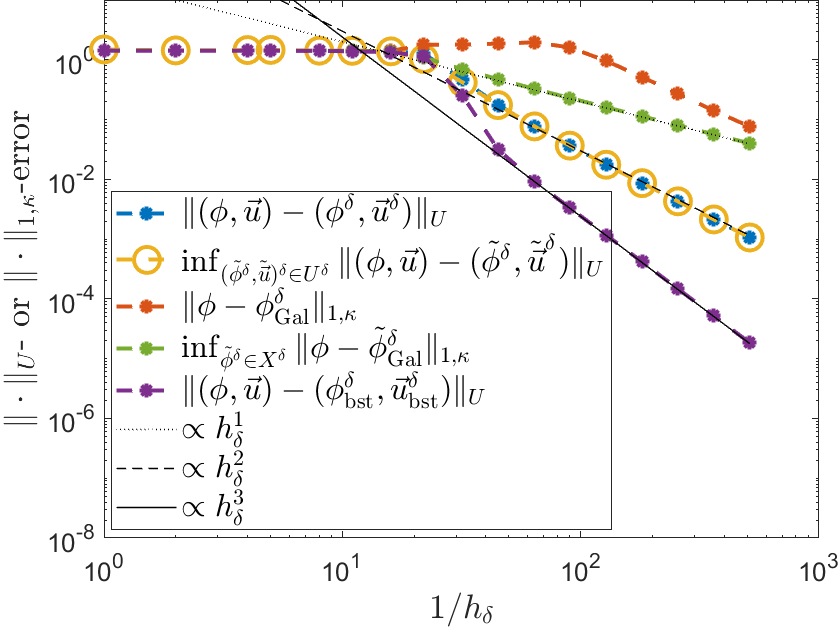}
\end{subfigure}
\begin{subfigure}{.5\textwidth}
\centering
\includegraphics[width=\linewidth]{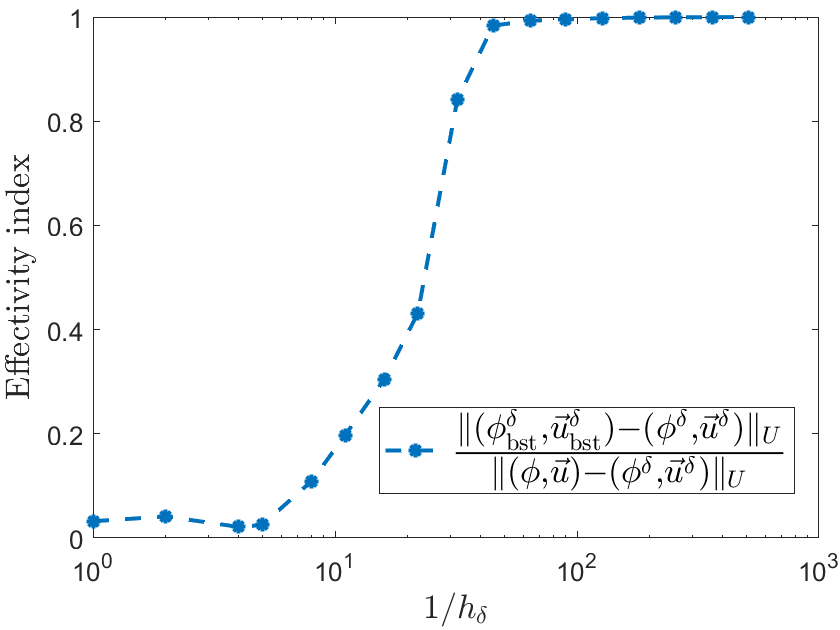}
\end{subfigure}
\caption{Errors in piecewise linear FOSLS and Galerkin approximations of a plane wave solution.}
\label{fig:errorsunitsquare}
\end{figure}
show that in both $\|\cdot\|_{L_2(\Omega)}$- and $\|\cdot\|_U$-norm the FOSLS solution is as good as the best approximation, whereas the Galerkin solution suffers from pollution.
The asymptotic rate in $\|\cdot\|_U$-norm (or $\|\cdot\|_{1,\kappa}$-norm) is better for the FOSLS method. 
This can be explained by the fact that with the FOSLS method also the gradient of the solution is approximated by piecewise linears.
The boosted approximation converges with a better asymptotic rate than the FOSLS method, and, consequently, the a posteriori error estimate is asymptotically exact.
When \new{the} number of degrees of freedom in $V_\mp^\delta$ per wavelength is, however, less than $2$, the boosted solution cannot provide a meaningful approximation, and so the error estimator is of no use.

Finally, to get an impression of the difference in accuracy of the FOSLS and Galerkin solutions, in Figure~\ref{fig:pics}
\begin{figure}[h!]
\hspace*{0cm}
\begin{subfigure}{.5\textwidth}
\centering
\includegraphics[width=\linewidth]{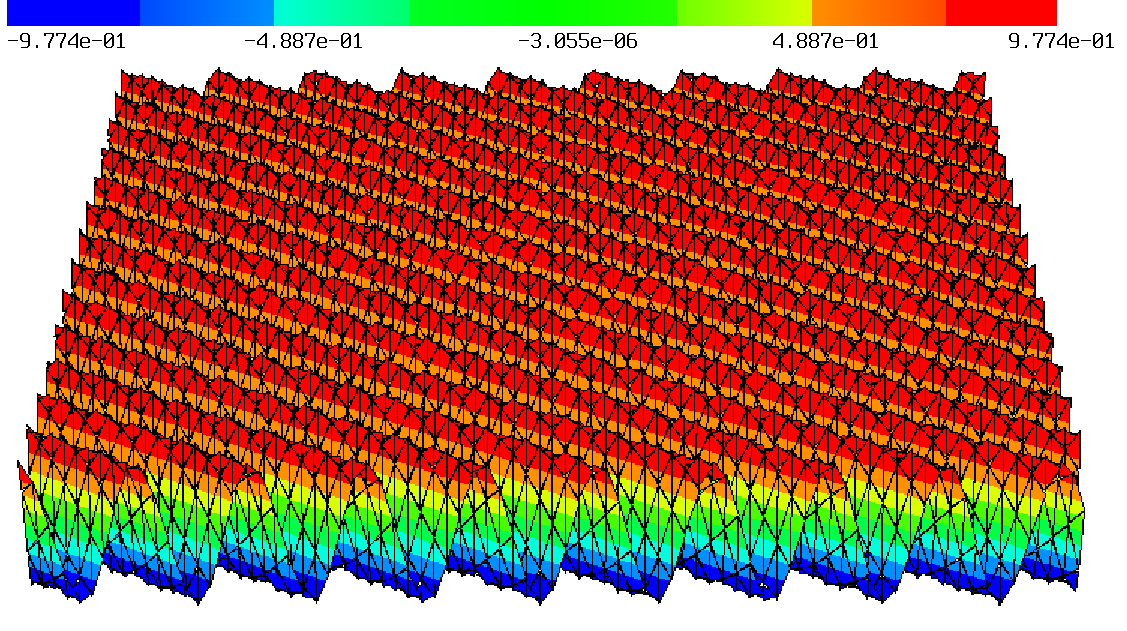}
\end{subfigure}\hspace*{0cm}%
\begin{subfigure}{.5\textwidth}
\centering
\includegraphics[width=\linewidth]{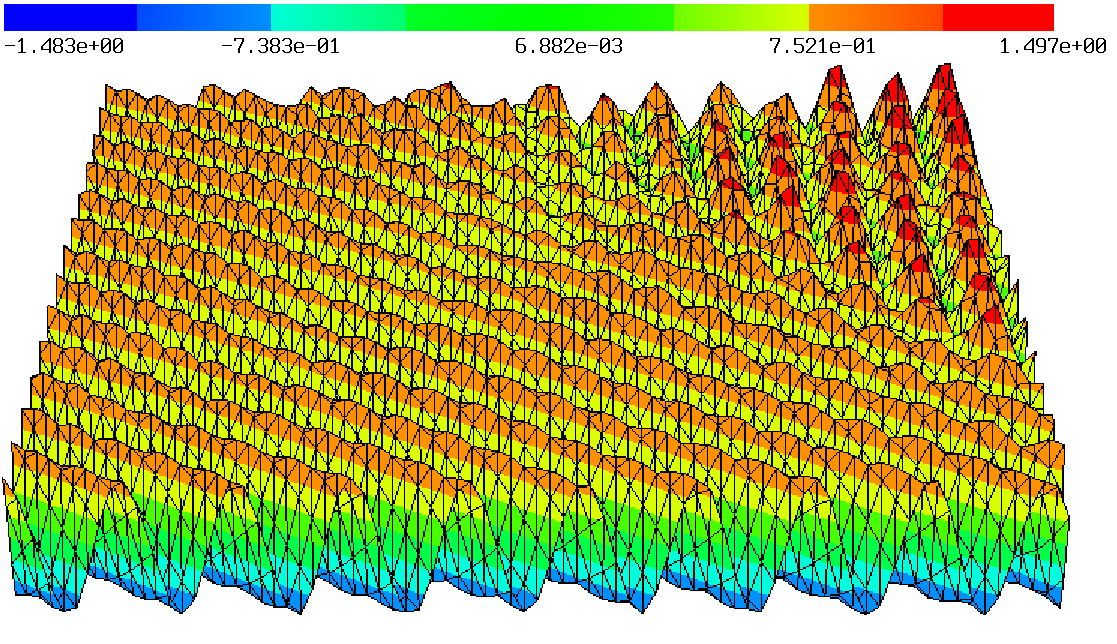}
\end{subfigure}
\caption{Piecewise linear FOSLS and Galerkin solution for $h_\delta = 1/80$. 
Left: $\Re \phi^\delta$, right: $\Re \phi_{\rm Gal}^\delta$.}
\label{fig:pics}
\end{figure}
we show $\Re \phi^\delta$ and $\Re \phi_{\rm Gal}^\delta$ for $h_\delta=\frac{1}{80}$.

\subsection{Scattering problem on a non-trapping domain}
We now consider an example from \cite{38.25}. Namely, let $\Omega:= (0,1)^2\setminus \overline{D}$, $\Gamma_D:=\partial D$, and $\Gamma_R:=\partial (0,1)^2$, where
$$
D:=\{\vec{x}\in (0,1)^2 \colon 2|x_1|-\tfrac12<x_2<|x_1|\}.
$$
$f=g_D=0$ and 
$g=\frac{1}{\kappa}(\frac{1}{\kappa} \nabla \phi_{\kappa \vec{r}}\cdot\vec{n}\pm i \phi_{\kappa \vec{r}})|_{\Gamma_R}=\frac{-i (\vec{r}\cdot\vec{n}\mp 1)}{\kappa}\phi_{\kappa \vec{r}}|_{\Gamma_R}$, where 
$\phi_{\kappa \vec{r}}$ is as defined in the previous subsection again for $\vec{r}=(\cos(\pi/3),\sin(\pi/3))$.
This problem models the (soft) scattering of an incoming wave $\phi_{\kappa \vec{r}}$ by the obstacle $D$. We take $\kappa=30 \pi$.

We consider a sequence of triangulations $(\mathcal{T}^\delta)_{\delta\in\Delta}$ of $\overline{\Omega}$, where each triangulation is created from its predecessor by a newest vertex bisection.
The initial triangulation is created using \verb+NGSolve+ and has 14 triangles and is shown in the left picture in Figure~\ref{fig:domainplusmesh}.
We take $U^\delta=\cS_p^{0}(\tria^\delta) \times \cS_p^{0}(\tria^\delta)^{2}$,
$V_\mp^\delta=(\cS_{p+2}^{0}(\tria^\delta) \times \RT_{p+2}(\tria^\delta)) \cap V_\mp$, and $X^\delta=\cS_p^{0}(\tria^\delta)$ here for $p=3$.
We consider both uniform refinement and, for the FOSLS method, adaptive refinements. The latter are driven by the a posteriori error estimator presented in Sect.~\ref{sec:boosted}, using D\"{o}rfler marking with parameter $\theta=0.6$.
In Figure~\ref{fig:domainplusmesh}, we show the initial and an adaptively refined mesh.
\begin{figure}[h!]
\begin{subfigure}{.5\textwidth}
\centering
\includegraphics[width=\linewidth]{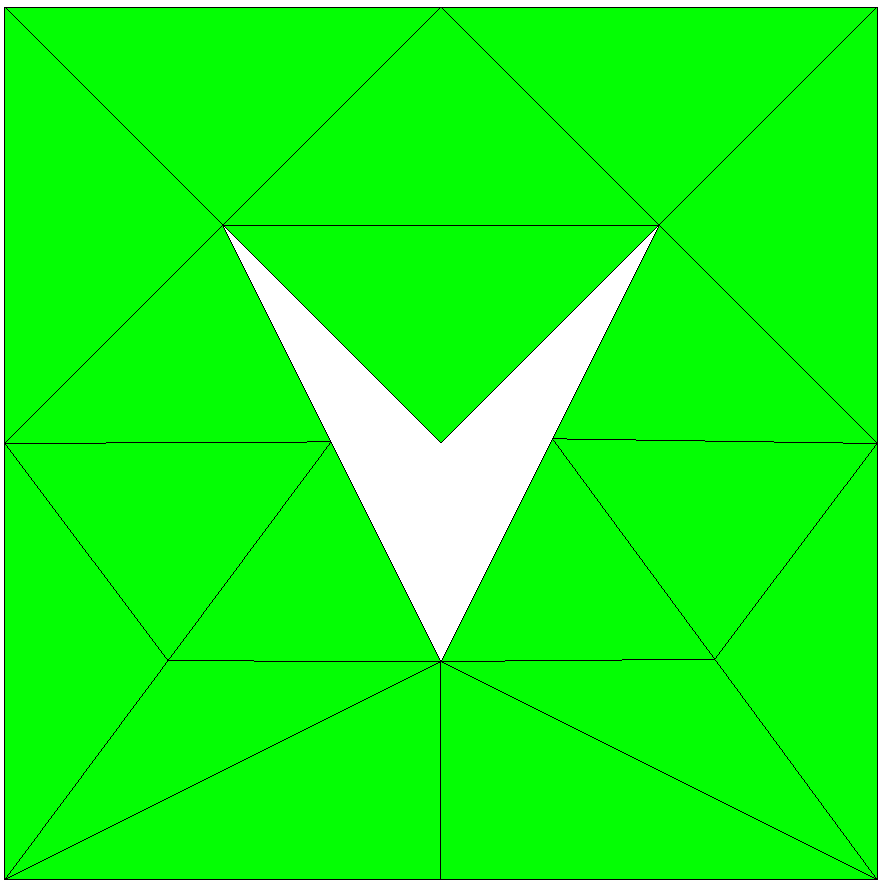}
\end{subfigure}\hspace*{0cm}%
\begin{subfigure}{.5\textwidth}
\centering
\includegraphics[width=\linewidth]{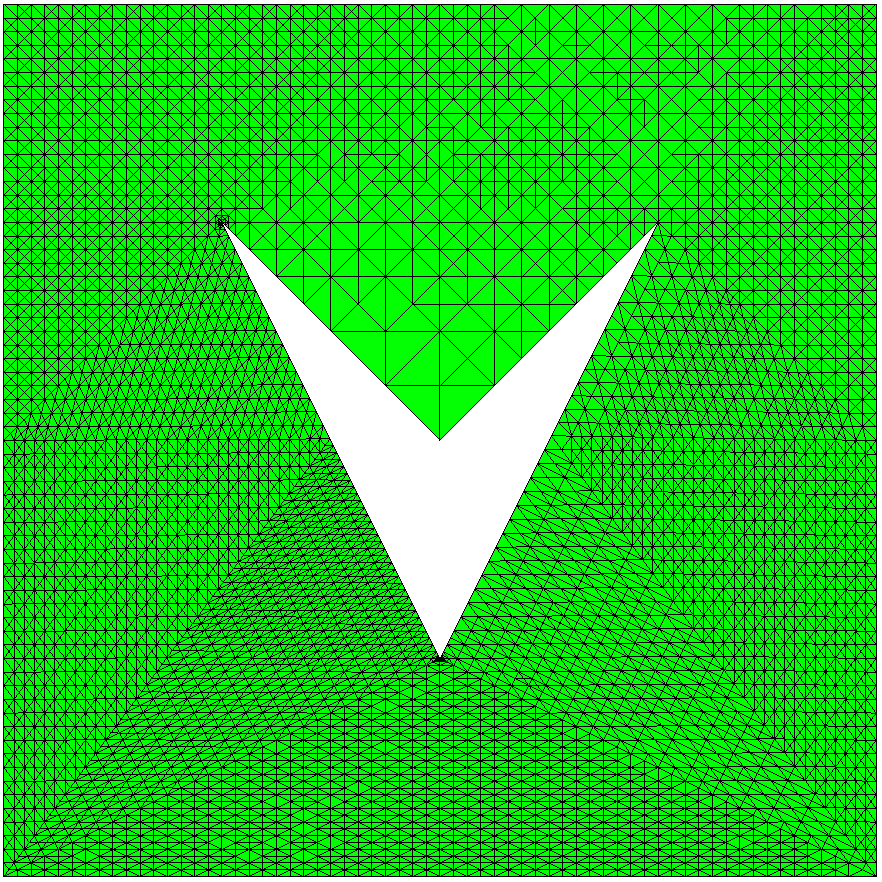}
\end{subfigure}
\caption{Non-trapping domain: left: Initial mesh, right: adaptively refined mesh with $12595$ triangles.}
\label{fig:domainplusmesh}
\end{figure}
In Figure~\ref{fig:error-pollution-effectivity}, we show the $\|\cdot\|_U$- or $\|\cdot\|_{1,\kappa}$-error in the (uniform and adaptive) FOSLS- or Galerkin-solutions, the pollution factors, 
and for the FOSLS case, the effectivity index of the a posteriori error estimator. To compute these quantities, we have replaced the unknown exact solution $(\phi,\vec{u})$ by the boosted FOSLS solution with trial- and test-spaces $\cS_4^{0}(\tilde{\tria}^{\new{\delta}})\times \cS_4^{0}(\tilde{\tria}^{\new{\delta}})^2$ and
$(\cS_{6}^{0}(\tilde{\tria}^{\new{\delta}}) \times \RT_{6}(\tilde{\tria}^{\new{\delta}})) \cap V_\mp$ w.r.t.~a sufficiently fine adaptively refined partition $\tilde{\tria}^{\new{\delta}}$.
\begin{figure}[h!]
\begin{subfigure}{.5\textwidth}
\centering
\includegraphics[width=\linewidth]{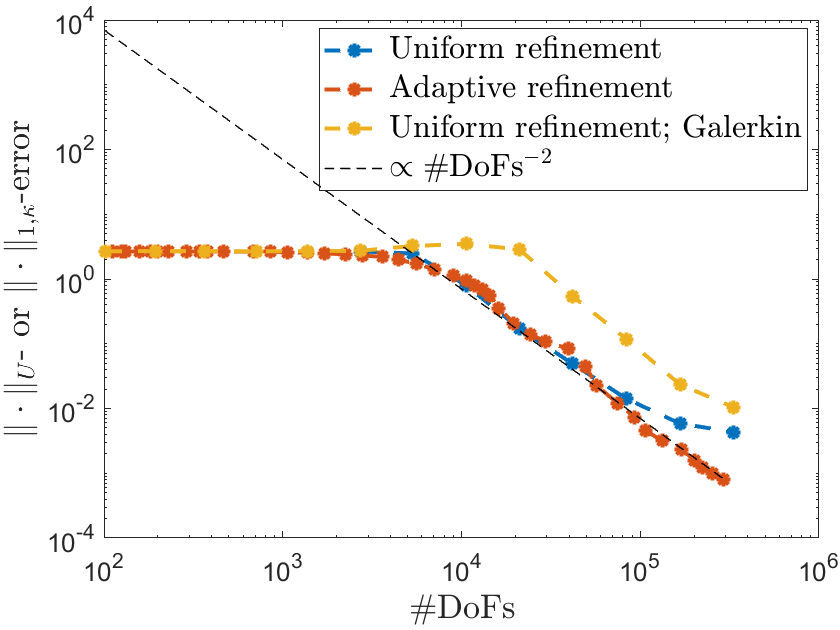}
\end{subfigure}\hspace*{0cm}%
\begin{subfigure}{.5\textwidth}
\centering
\includegraphics[width=\linewidth]{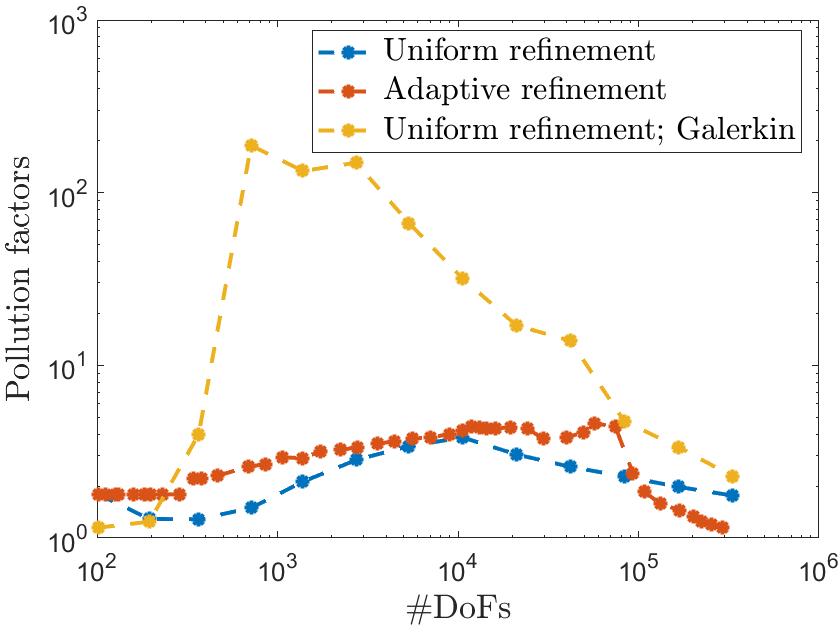}
\end{subfigure}
\begin{subfigure}{.5\textwidth}
\centering
\includegraphics[width=\linewidth]{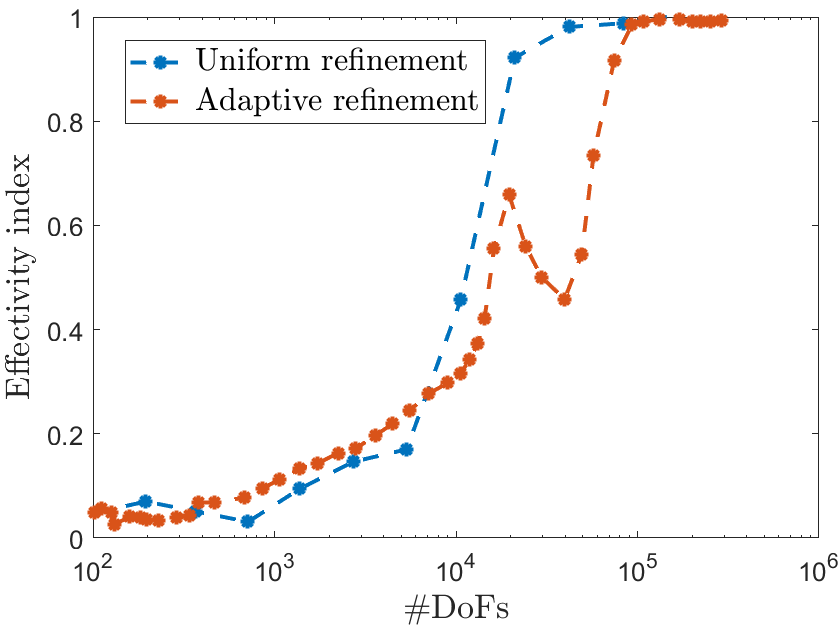}
\end{subfigure}
\caption{Errors, pollution factors, and effectivity indices for scattering example on non-trapping domain with approximation by piecewise cubics.}
\label{fig:error-pollution-effectivity}
\end{figure}

The numerical results show that the best possible rate is achieved with the adaptive FOSLS method. For this problem, the advantage of adaptivity is only visible for $\#$DoFs$>10^5$. It can be expected that for smaller $\kappa$, and so a smaller global wavelength, adaptivity is already advantageous for smaller problem sizes.
The FOSLS method has a smaller pollution factor than the Galerkin method, but the difference is smaller than with the piecewise linear trial spaces employed in the previous two subsections. 
When keeping the piecewise  cubic trial spaces for FOSLS and Galerkin methods, this difference will increase for increasing $\kappa$ when the order of the test space for the FOSLS method is (slowly) increased as well.

\subsection{Scattering problem on a trapping domain}
Again we consider an example from \cite{38.25}, where
$\Omega = (0,1)^2\setminus \overline{D}$, $\Gamma_D=\partial D$, and $\Gamma_R=\partial (0,1)^2$ and $D=D_1\cup D2$ where
$$
D_1=\{\vec{x}\in (0,1)^2 \colon \tfrac{1}{4}\leq|y|\leq\tfrac{1}{4}-\tfrac{1}{2}x\}.
$$
$$
D_2=\{\vec{x}\in (0,1)^2 \colon \tfrac{1}{2}+2x\leq |y|\leq \tfrac{1}{2}; x\geq -\tfrac{1}{2}\}.
$$
We set $f=g_D=0$ and 
$g=\frac{1}{\kappa}(\frac{1}{\kappa} \nabla \phi_{\kappa \vec{r}}\cdot\vec{n}\pm i \phi_{\kappa \vec{r}})|_{\Gamma_R}=\frac{-i (\vec{r}\cdot\vec{n}\mp 1)}{\kappa}\phi_{\kappa \vec{r}}|_{\Gamma_R}$, where 
$\phi_{\kappa \vec{r}}$ is as in the previous subsection but now with $\vec{r}=(\cos(9\pi/10),\sin(9\pi/10))$, \new{and $\kappa=30\pi$.} 

The domain, and the initial and an adaptively refined triangulation are illustrated in Figure~\ref{fig:trapdomainplusmesh}.
\begin{figure}[h!]
\begin{subfigure}{.5\textwidth}
\centering
\includegraphics[width=\linewidth]{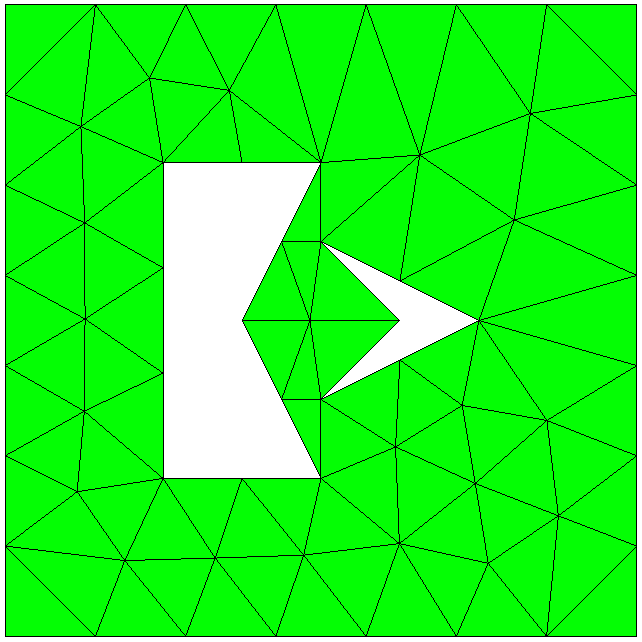}
\end{subfigure}\hspace*{0cm}%
\begin{subfigure}{.5\textwidth}
\centering
\includegraphics[width=\linewidth]{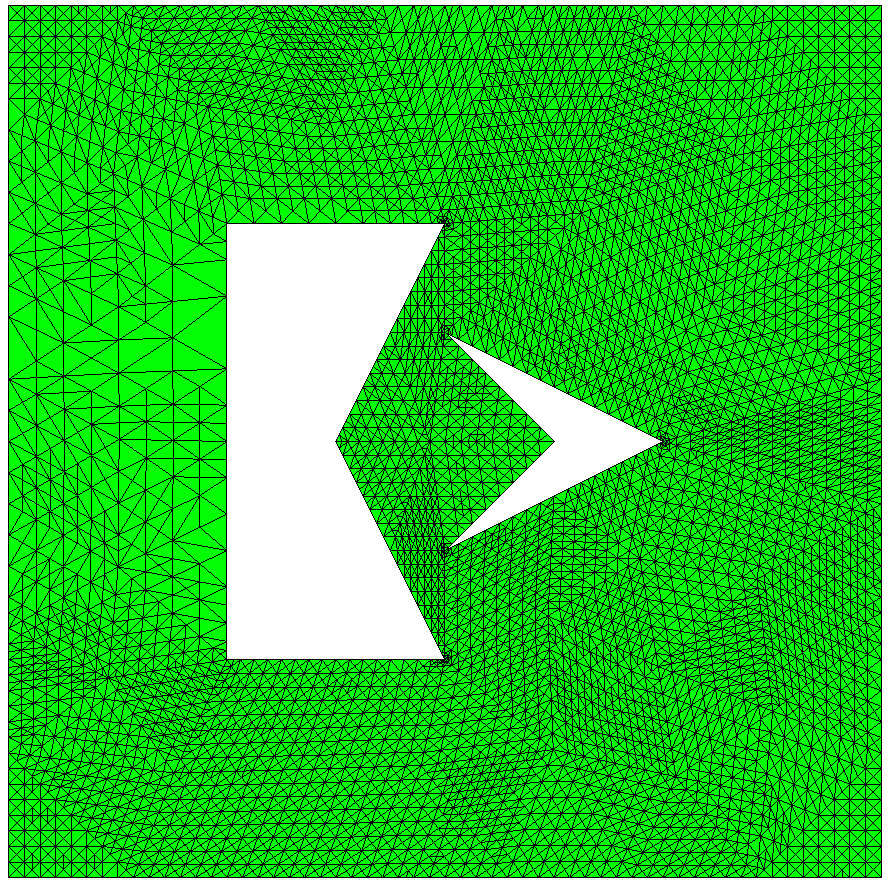}
\end{subfigure}
\caption{Trapping domain; left: Initial mesh, right: adaptively refined mesh with $12329$ triangles.}
\label{fig:trapdomainplusmesh}
\end{figure}

In Figure~\ref{fig:traperror-pollution-effectivity}, we show the $\|\cdot\|_U$- or $\|\cdot\|_{1,\kappa}$-error in the (uniform and adaptive) FOSLS- or Galerkin-solution, the pollution factors, 
and for the FOSLS case, the effectivity index of the a posteriori error estimator. To compute these quantities, we have replaced the unknown exact solution $(\phi,\vec{u})$ by 
a sufficiently accurate boosted higher order FOSLS approximation.
\begin{figure}[h!]
\begin{subfigure}{.5\textwidth}
\centering
\includegraphics[width=\linewidth]{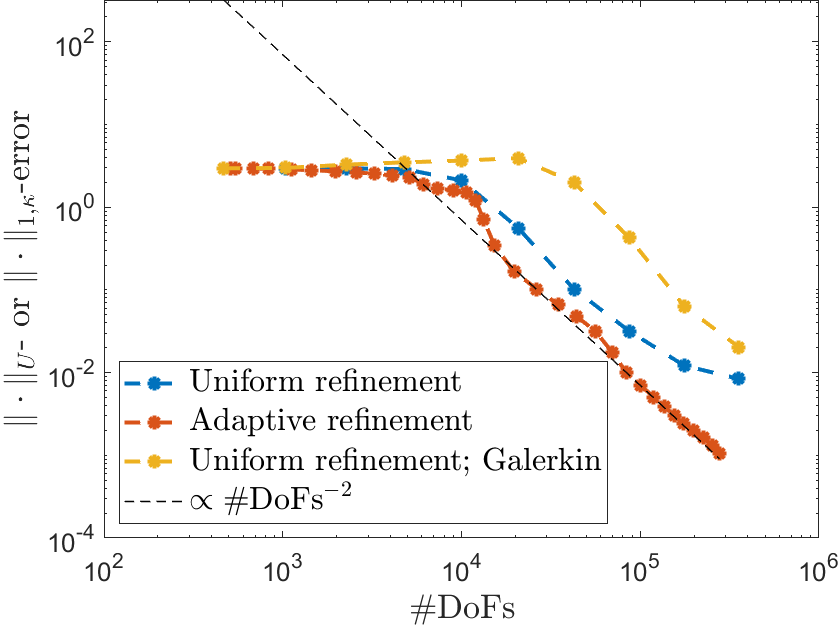}
\end{subfigure}\hspace*{0cm}%
\begin{subfigure}{.5\textwidth}
\centering
\includegraphics[width=\linewidth]{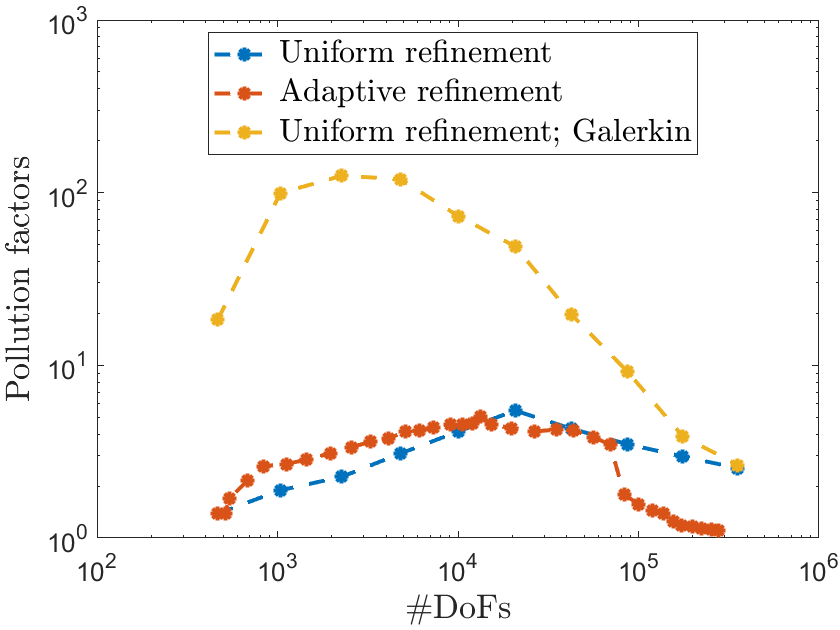}
\end{subfigure}
\begin{subfigure}{.5\textwidth}
\centering
\includegraphics[width=\linewidth]{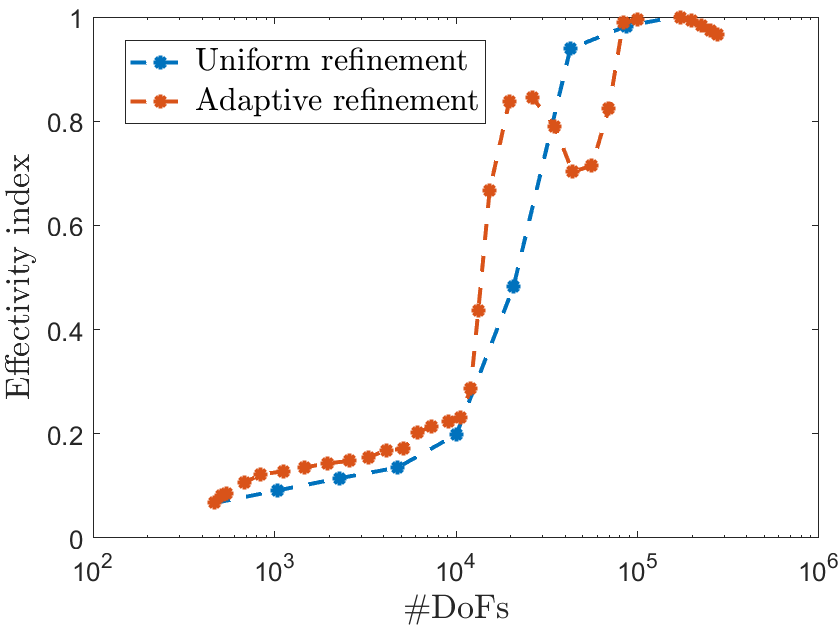}
\end{subfigure}
\caption{Errors, pollution factors, and effectivity indices for scattering example on trapping domain with approximation by piecewise cubics.}
\label{fig:traperror-pollution-effectivity}
\end{figure}
The results are not much different as those from the preceding subsection for a non-trapping domain.

\section{Conclusion \new{and outlook}} \label{sec:conclusions} 
We considered an ultra-weak first order system (FOSLS) formulation of the Helmholtz equation. By employing the optimal test-norm, the (unfeasible) exact least-squares minimization yields the best approximation from the trial space w.r.t.~the norm on $L_2(\Omega) \times L_2(\Omega)^d$.
The performance of the practical, implementable method, which uses a discretized dual norm, is governed by the computable inf-sup constant $\gamma_\kappa^\delta$ determined by $\kappa$ and the pair of trial- and test-spaces.
Its reciprocal is the pollution factor of the method.
To show the uniform boundedness of $\gamma_\kappa^\delta$ from below we used its characterization in terms the approximability from the test space of the solution of the adjoint problem with a forcing term from the trial space. 
Using results from \cite{202.8}, we showed that on convex polygons the FOSLS method is pollution-free when
\new{$\max(\log \kappa,p^2)/\tilde{p}$ is sufficiently small, where $p$ and $\tilde{p}$ are the orders of the finite element spaces at trial- and test side.}

The (approximate) Riesz' lift of the residual can be used for a posteriori error estimation, as a local error indicator for adaptive refinement, and to construct an improved approximation. In experiments, this `boosted' approximation converges at an improved rate, so that the estimator is asymptotically exact.

To compare our results with those by the standard Galerkin discretization, we constructed a computable quantity that is a lower
bound, and up to a constant factor, an upper bound for the pollution factor of this method.
Our numerical experiments on convex or non-convex, including trapping domains show that our FOSLS method has smaller pollution factors, and yields more accurate approximations than the Galerkin method.

So far, we have solved the linear system that arises from the FOSLS method with a direct solver. We are, however, optimistic that its Hermitian saddle-point form provides advantages when it concerns an \new{iterative solution. Indeed, since for trial- and test-spaces that yield pollution-free approximations the
 Schur complement corresponds to a uniformly boundedly invertible operator on $L_2(\Omega)^{d+1}$, a fast iterative solution only requires a good preconditioner for the positive definite upper left block. This block is the representation of a least squares discretisation of Helmholtz in first order form. Good preconditioners for this block might be easier to construct than such preconditioners for the indefinite matrix resulting from the common Galerkin discretisation. This will be \new{a} topic of future study.}

\end{document}